\documentclass[11pt]{amsart}
\usepackage{graphicx,subcaption}
\usepackage{amsmath} 
\usepackage{amsthm,amsfonts,amssymb,mathrsfs,amscd,amstext,amsbsy} 
\usepackage{epic,eepic} 
\usepackage{yfonts}
\usepackage{paralist,enumerate}
\usepackage[all]{xy}
\usepackage{hyperref}
\hypersetup{colorlinks}

\usepackage{tikz}
\usetikzlibrary {positioning}

\newcommand{\D}{\displaystyle}

\newtheorem{theorem}{Theorem}[section]

\newtheorem{lemma}[theorem]{Lemma}

\newtheorem{theo}[theorem]{Theorem}
\newtheorem{lem}[theorem]{Lemma}

\newtheorem{rem}[theorem]{Remark}
\newtheorem{exa}[theorem]{Example}

\newtheorem{definition}[theorem]{Definition}
\newtheorem*{Definition*}{Definition}
\numberwithin{equation}{section}

\def\qed{\hfill \ifhmode\unskip\nobreak\fi\quad\ifmmode\Box\else$\Box$\fi\\ }

\begin{document}

\title[Almost complex $S^1$-manifold with 4 fixed points]{Circle actions on almost complex manifolds with 4 fixed points}
\author{Donghoon Jang}
\address{School of Mathematics, Korea Institute for Advanced Study, 85 Hoegiro, Dongdaemun-gu, Seoul, 02455, Korea}
\email{groupaction@kias.re.kr}

\begin{abstract}
Let the circle act on a compact almost complex manifold $M$. In this paper, we classify the fixed point data of the action if there are 4 fixed points and the dimension of the manifold is at most 6. First, if $\dim M=2$, then $M$ is a disjoint union of rotations on two 2-spheres. Second, if $\dim M=4$, we prove that the action alikes a circle action on a Hirzebruch surface. Finally, if $\dim M=6$, we prove that six types occur for the fixed point data; $\mathbb{CP}^3$ type, complex quadric in $\mathbb{CP}^4$ type, Fano 3-fold type, $S^6 \cup S^6$ type, and two unknown types that might possibly be realized as blow ups of a manifold like $S^6$. When $\dim M=6$, we recover the result by Ahara \cite{A} in which the fixed point data is determined if furthermore $\mathrm{Todd}(M)=1$ and $c_1^3(M)[M] \neq 0$, and the result by Tolman \cite{T} in which the fixed point data is determined if furthermore the base manifold admits a symplectic structure and the action is Hamiltonian.
\end{abstract}

\maketitle

\tableofcontents

\section{Introduction}

$\indent$

The main purpose of this paper is to classify the fixed point data of a circle action on a compact almost complex manifold when there are four fixed points and the dimension of the manifold is at most six. An \emph{almost complex manifold} $(M,J)$ is a manifold $M$ with a smooth linear complex structure $J$ on each tangent space of $M$. From this it follows that the dimension of an almost complex manifold is necessarily even. Examples of almost complex manifolds are complex manifolds, symplectic manifolds, and K\"ahler manifolds. Therefore, our results will apply to those manifolds. An action of a group $G$ on an almost complex manifold $(M,J)$ is said to preserve the almost complex structure if the action satisfies $dg \circ J=J \circ dg$ for any $g \in G$. Let the circle $S^1$ act on a compact almost complex manifold $(M,J)$. Throughout this paper, we assume that the action preserves the almost complex structure $J$. If $p$ is an isolated fixed point, then local action of $S^1$ near $p$ is described as $g \cdot (z_1,\cdots,z_n)=(g^{w_p^1} z_1, \cdots, g^{w_p^n} z_n)$, where $g \in S^1 \subset \mathbb{C}$ and $w_p^i$ are non-zero integers, for $1 \leq i \leq n$. The $w_p^i$ are called \emph{weights} at $p$ (also called \emph{rotation numbers}). By the fixed point data, we mean a collection of multisets of weights at the fixed points. Finally, the fixed point data encodes some of the invariants of $M$; for instance it encodes the Hirzebruch $\chi_y$-genus (and hence the Euler characteristic, the signature, and the Todd genus) and Chern numbers of $M$. In addition, if the base manifold $M$ admits a symplectic structure and the action is Hamiltonian, then we can further recover (equivariant) Chern classes and (equivariant) cohomology of $M$. Therefore, to classify a manifold, one may classify the fixed point data.

Assume that the fixed point set are discrete. In this case, we can associate a multigraph to $M$, where vertices are the fixed points. Roughly speaking, edges are drawn in the following way: if a fixed point $p$ has weight $w>0$, then there exists another fixed point $p'$ that has weight $-w$; an edge $\epsilon$ is then drawn from $p$ to $p'$ with label $w$. Then the classification of the fixed point data can be done by considering possible multigraphs to $M$. The use of such multigraphs has been used in the literature; for instance, see \cite{A}(implicitly) \cite{GS}, \cite{JT}, \cite{T}. In this paper, we consider a special multigraph; see section 3 (Lemma \ref{l35}) for how we assign a multigraph to $M$.

Let the circle act on a compact almost complex manifold $M$. First, if there is one fixed point, then $M$ must be a point. Second, if there are two fixed points, then $M$ is either $S^2$, or $\dim M=6$ and the fixed point data is the same as a rotation on $S^6$ \cite{J3}, \cite{K1}, \cite{PT}. Third, if there are three fixed points, then the dimension of $M$ must be four and the fixed point data is the same as a standard $S^1$-action on $\mathbb{CP}^2$. This is announced in \cite{J3}, whose proof is essentially the same as one in \cite{J1} for symplectic case after careful treatments. However, the proof in \cite{J1} requires a lot of computations even when there are three fixed points, and this predicts that the more fixed points there are, the harder computations are. Kosniowski conjectures that the dimension of a manifold is bounded above by a linear function of the number of fixed points; it is conjectured that $\dim M \leq 4k$, where $k$ is the number of fixed points \cite{K2}. By the discussion above, the conjecture holds if there are at most three fixed points.

Now, assume that there are four fixed points. If $\dim M=2$, then $M$ is a disjoint union of rotations on two 2-spheres. This can be seen by the fact that among oriented Riemann surfaces, only the 2-sphere $S^2$ supports a circle action with fixed points, that is, a rotation. If $\dim M=4$, we shall see that the action alikes a circle action on a Hirzebruch surface; see Example \ref{e210}.

\begin{theo} \label{t11} Let the circle act on a 4-dimensional compact almost complex manifold with 4 fixed points. Then the multisets of weights are $\{a,b\}$, $\{-a,b\}$, $\{-b,c\}$, and $\{-b,-c\}$ for some positive integers $a,b,c$ such that either $a \equiv c \mod b$ or $a \equiv -c \mod b$. \end{theo}

\begin{figure}
\centering
\begin{subfigure}[b][6cm][s]{.45\textwidth}
\centering
\vfill
\begin{tikzpicture}[state/.style ={circle, draw}]
\node[state] (A) {$p_1$};
\node[state] (B) [above left=of A] {$p_2$};
\node[state] (C) [above right=of A] {$p_3$};
\node[state] (D) [above right=of B] {$p_4$};
\path (A) [->] edge node[right] {$a$} (B);
\path (A) [->] edge node [right] {$b$} (C);
\path (B) [->] edge node [right] {$b$} (D);
\path (C) [->] edge node [right] {$c$} (D);
\end{tikzpicture}
\vfill
\caption{Case 1: $a \equiv c \mod b$}\label{fig1-1}
\vspace{\baselineskip}
\end{subfigure}\qquad
\begin{subfigure}[b][6cm][s]{.45\textwidth}
\centering
\vfill
\begin{tikzpicture}[state/.style ={circle, draw}]
\node[state] (A) {$p_1$};
\node[state] (B) [above left=of A] {$p_2$};
\node[state] (C) [above =of B] {$p_3$};
\node[state] (D) [above right=of C] {$p_4$};
\path (A) [->] edge node[right] {$a$} (B);
\path (A) [->] edge node [right] {$b$} (D);
\path (B) [->] edge node [right] {$b$} (C);
\path (C) [->] edge node [right] {$c$} (D);
\end{tikzpicture}
\vfill
\caption{Case 2: $a \equiv -c \mod b$}\label{fig1-2}
\end{subfigure}
\caption{Multigraphs for Theorem \ref{t11}.}\label{fig1}
\end{figure}
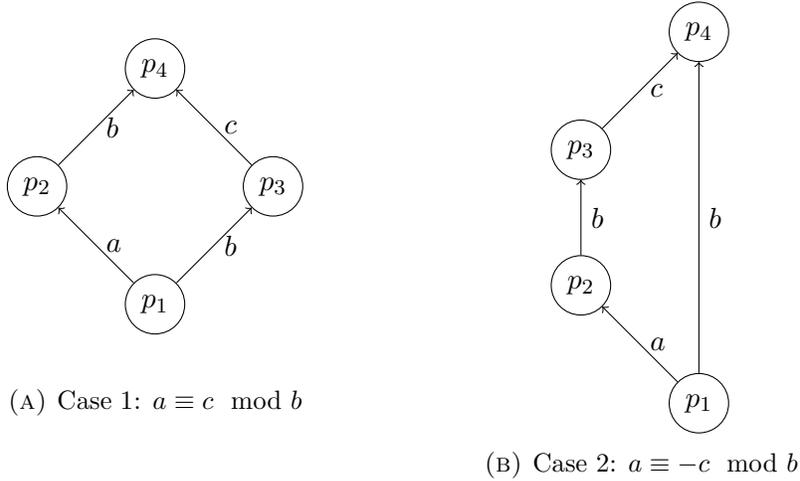

In Theorem \ref{t11}, in the case that $a \equiv c \mod b$, the multigraph associated to $M$ is provided in Figure \ref{fig1-1}. In the case that $a \equiv -c \mod b$, the corresponding multigraph is Figure \ref{fig1-2}.

Assume that the dimension of $M$ is six. In \cite{A}, Ahara considers the classification of the fixed point data if in addition, $M$ satisfies $\mathrm{Todd}(M)=1$ and $c_1^3(M)[M] \neq 0$, where $\mathrm{Todd}(M)$ is the Todd genus, $c_1(M)$ is the first Chern class, and $[M]$ is the fundamental class of $M$. Note that the Todd genus is the genus belonging to the power series $\displaystyle \frac{x}{1-e^{-x}}$. In this case, Ahara shows that three types of the fixed point data occur; $\mathbb{CP}^3$ type, complex quadric in $\mathbb{CP}^4$ type, and Fano 3-fold type. In \cite{T}, Tolman classifies the fixed point data if $M$ admits a symplectic structure and the action is Hamiltonian. In this paper, we shall discuss the problem without those additional assumptions. Therefore our result covers the two results known. When $\dim M=6$, our result is the following:

\begin{theo} \label{t12} Let the circle act on a 6-dimensional compact almost complex manifold with 4 fixed points. Then exactly one of the following holds:
\begin{enumerate}[(1)]
\item $\mathrm{Todd}(M)=1$ and the multisets of weights are $\{a,b,c\}$, $\{-a,b-a,c-a\}$, $\{-b,a-b,c-b\}$, $\{-c,a-c,b-c\}$ for some mutually distinct positive integers $a,b,c$.
\item $\mathrm{Todd}(M)=1$ and the multisets of weights are $\{a,a+b,a+2b\}$, $\{-a,b,a+2b\}$, $\{-a-2b,-b,a\}$, $\{-a-2b,-a-b,-a\}$ for some positive integers $a,b$.
\item $\mathrm{Todd}(M)=1$ and the multisets of weights are $\{1,2,3\}$, $\{-1,1,a\}$, $\{-1,-a,1\}$, $\{-1,-2,-3\}$ for some positive integer $a$. 
\item $\mathrm{Todd}(M)=0$ and the multisets of weights are $\{-a-b,a,b\}$, $\{-c-d,c,d\}$, $\{-a,-b,a+b\}$, $\{-c,-d,c+d\}$ for some positive integers $a,b,c,d$.
\item $\mathrm{Todd}(M)=0$ and the multisets of weights are $\{-3a-b,a,b\}$, $\{-2a-b,3a+b,3a+2b\}$, $\{-a,-a-b,2a+b\}$, $\{-b,-3a-2b,a+b\}$ for some positive integers $a,b$, by reversing the circle action if necessary.
\item $\mathrm{Todd}(M)=0$ and the multisets of weights are $\{-a-b,2a+b,b\}$, $\{-2a-b,a,b\}$, $\{-b,-2a-b,a+b\}$, $\{-a,-b,2a+b\}$ for some positive integers $a,b$.
\end{enumerate} \end{theo}

\begin{figure}
\centering
\begin{subfigure}[b][7cm][s]{.3\textwidth}
\centering
\vfill
\begin{tikzpicture}[state/.style ={circle, draw}]
\node[state] (a) {$p_1$};
\node[state] (b) [above=of a] {$p_2$};
\node[state] (c) [above=of b] {$p_3$};
\node[state] (d) [above=of c] {$p_4$};
\path (a) [->] edge node[right] {$a$} (b);
\path (a) [->] [bend left =40] edge node [right] {$b$} (c);
\path (a) [->] [bend left =40]edge node [right] {$c$} (d);
\path (b) [->] edge node [right] {$b-a$} (c);
\path (b) [->] [bend right =40]  edge node [right] {$c-a$} (d);
\path (c) [->] edge node [right] {$c-b$} (d);
\end{tikzpicture}
\vfill
\caption{$\mathbb{CP}^3$ type}\label{fig2-1}
\end{subfigure}
\begin{subfigure}[b][7cm][s]{.3\textwidth}
\centering
\vfill
\begin{tikzpicture}[state/.style ={circle, draw}]
\node[state] (a) {$p_1$};
\node[state] (b) [above=of a] {$p_2$};
\node[state] (c) [above=of b] {$p_3$};
\node[state] (d) [above=of c] {$p_4$};
\path (a) [->] edge node[right] {$a$} (b);
\path (a) [->] [bend left =40] edge node [left] {$a+2b$} (c);
\path (a) [->] [bend left =40]edge node [left] {$a+b$} (d);
\path (b) [->] edge node [right] {$b$} (c);
\path (b) [->] [bend right =40]  edge node [right] {$a+2b$} (d);
\path (c) [->] edge node [right] {$a$} (d);
\end{tikzpicture}
\vfill
\caption{Complex quadric}\label{fig2-2}
\vspace{\baselineskip}
\end{subfigure}\qquad
\begin{subfigure}[b][7cm][s]{.3\textwidth}
\centering
\vfill
\begin{tikzpicture}[state/.style ={circle, draw}]
\node[state] (a) {$p_1$};
\node[state] (b) [above=of a] {$p_2$};
\node[state] (c) [above=of b] {$p_3$};
\node[state] (d) [above=of c] {$p_4$};
\path (a) [->] [bend left =40] edge node[right] {$2$} (d);
\path (a) [->] [bend left =20] edge node [right] {$3$} (d);
\path (a) [->] edge node [right] {$1$} (b);
\path (b) [->] [bend left=20] edge node [right] {$1$} (c);
\path (b) [->] [bend right =20]  edge node [right] {$a$} (c);
\path (c) [->] edge node [right] {$1$} (d);
\end{tikzpicture}
\vfill
\caption{Fano 3-fold type}\label{fig2-3}
\vspace{\baselineskip}
\end{subfigure}\qquad
\begin{subfigure}[b][3.5cm][s]{.27\textwidth}
\centering
\vfill
\begin{tikzpicture}[state/.style ={circle, draw}]
\node[state] (a) {$p_1$};
\node[state] (b) [right=of a] {$p_2$};
\node[state] (c) [above=of a] {$p_3$};
\node[state] (d) [right=of c] {$p_4$};
\path (a) [->] [bend left =10]edge node[left] {$a$} (c);
\path (a) [->] [bend left =40]edge node[left] {$b$} (c);
\path (c) [->] [bend left =10]edge node[right] {$a+b$} (a);
\path (b) [->] edge node[left] {$d$} (d);
\path (b) [->] [bend left =30]edge node[left] {$c$} (d);
\path (d) [->] [bend left =20]edge node[right] {$c+d$} (b);
\end{tikzpicture}
\vfill
\caption{$S^6 \cup S^6$ type}\label{fig2-4}
\vspace{\baselineskip}
\end{subfigure}\qquad
\begin{subfigure}[b][3.5cm][s]{.27\textwidth}
\centering
\vfill
\begin{tikzpicture}[state/.style ={circle, draw}]
\node[state] (a) {$p_1$};
\node[state] (b) [right=of a] {$p_2$};
\node[state] (c) [above=of a] {$p_3$};
\node[state] (d) [right=of c] {$p_4$};
\path (b) [->] edge node[below] {$3a+b$} (a);
\path (a) [->] edge node[pos=.2, right, sloped, rotate=270] {$b$} (d);
\path (a) [->]edge node [left] {$a$} (c);
\path (b) [->] edge node [right] {$3a+2b$} (d);
\path (c) [->]edge node[pos=.3, right, sloped, rotate=270] {$2a+b$} (b);
\path (d) [->]edge node [above] {$a+b$} (c);
\end{tikzpicture}
\vfill
\caption{Unknown-I}\label{fig2-5}
\vspace{\baselineskip}
\end{subfigure}\qquad
\begin{subfigure}[b][3.5cm][s]{.27\textwidth}
\centering
\vfill
\begin{tikzpicture}[state/.style ={circle, draw}]
\node[state] (a) {$p_1$};
\node[state] (b) [right=of a] {$p_2$};
\node[state] (c) [above=of a] {$p_3$};
\node[state] (d) [right=of c] {$p_4$};
\path (a) [->] [bend right =20]edge node[left] {$b$} (c);
\path (a) [->] edge node[below] {$2a+b$} (b);
\path (c) [->] [bend right =20]edge node[left] {$a+b$} (a);
\path (b) [->] [bend left =20]edge node[right] {$a$} (d);
\path (b) [->] [bend right=20]edge node[right] {$b$} (d);
\path (d) [->] edge node[above] {$2a+b$} (c);
\end{tikzpicture}
\vfill
\caption{Unknown-II}\label{fig2-6}
\vspace{\baselineskip}
\end{subfigure}\qquad
\caption{Multigraphs for Theorem \ref{t12}.}\label{fig2}
\end{figure}
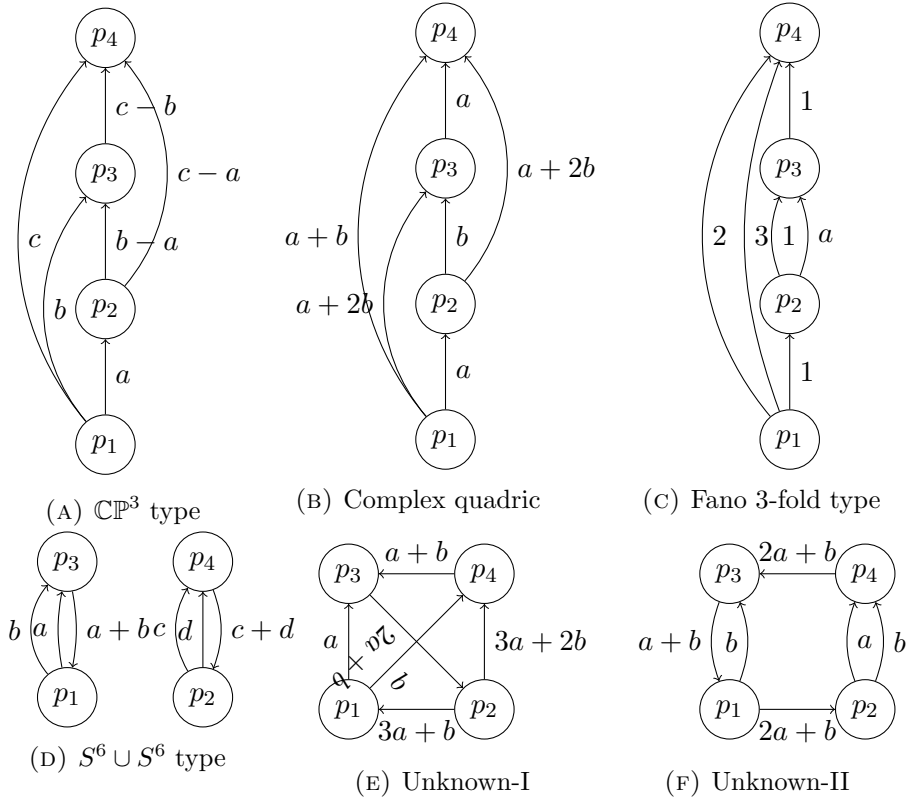

In each case, the multigraph associated to $M$ is given in Figure \ref{fig2}. In each case, let us call the manifold $\mathbb{CP}^3$ type, complex quadric in $\mathbb{CP}^4$ type, Fano 3-fold type, $S^6 \cup S^6$ type, unknown type I, and unknown type II, respectively. The reason is because in each case the fixed point data is the same as that of the manifold named with an appropriate action; see Examples \ref{e211}, \ref{e212}, \ref{e213}, and \ref{e214} for examples of manifolds of the first, second, third, and fourth type respectively. Note that an example of a manifold of the third type is known only for $a=4$ and 5. On the other hand, in the third case, if $M$ inherits a symplectic structure and the circle action preserves the symplectic structure, then the action is Hamiltonian and furthermore, the only possible values for $a$ are 4 and 5; see \cite{T}. Therefore, we have all examples of an action if $\mathrm{Todd}(M)=1$, $M$ admits a symplectic structure, and the action preserves the symplectic form. On the other hand, to the author's knowledge, a manifold with a fixed point data as the fifth type or the sixth type is not known. However, the fixed point data of the fifth type and the sixth type might be realized as a blowing up of a manifold $M$ at a fixed point or at an isotropy 2-sphere, where $M$ admits a circle action with 2 fixed points (like $S^6$). For the discussion, see Remark \ref{r215}.

The paper is organized as follows. For the description, consider a circle action on a compact almost complex manifold $M$ with a discrete fixed point set. In Section 2, we recall background needed and set up notations. In Section 3, we associate a labelled, directed multigraph to $M$. Our way of drawing edges is slightly different from others (see Lemma \ref{l35}), and this simplifies the proofs of Theorem \ref{t11} and Theorem \ref{t12}. In Section 4, we classify the fixed point data of $M$ when $\dim M=4$ and there are 4 fixed points. Section 5-6 are for the classification of the fixed point data of $M$ where $\dim M=6$ and there are 4 fixed points. When $\dim M=6$, either $\mathrm{Todd}(M)=1$ or $\mathrm{Todd}(M)=0$. These cases are dealt seperately in Section 5 and 6. Finally, the proof of Theorem \ref{t12} is given in Section 7.

\section{Background and Notation}

Let the circle act on a $2n$-dimensional almost complex manifold $M$. Let $p$ be an isolated fixed point. There are $n$ non-zero integers $w_p^i$, called \emph{weights}, associated to $p$, for $1 \leq i \leq n$. Denote by $\Sigma_p=\{w_p^1,\cdots,w_p^n\}$ the multiset of weights at $p$ and $n_p$ the number of negative weights at $p$. To prove Theorem \ref{t12}, we shall use the classification when the number of fixed points is at most three.

\begin{theorem} \label{t21} \cite{J3} Let the circle act on a compact almost complex manifold $M$.
\begin{enumerate}[(1)]
\item If there is one fixed point, then $M$ is a point.
\item If there are two fixed points, then either
\begin{enumerate}[(a)]
\item $M$ is the 2-sphere and weights at the fixed points are $a$ and $-a$ for some positive integer $a$, or
\item $\dim M=6$ and weights at the fixed points are $\{-a-b,a,b\}$ and $\{-a,-b,a+b\}$ for some positive integers $a$ and $b$.
\end{enumerate}
\item If there are three fixed points, then $\dim M=4$ and weights at the fixed points are $\{a+b,a\}$, $\{-a,b\}$, and $\{-b,-a-b\}$ for some positive integers $a$ and $b$.
\end{enumerate}
\end{theorem}

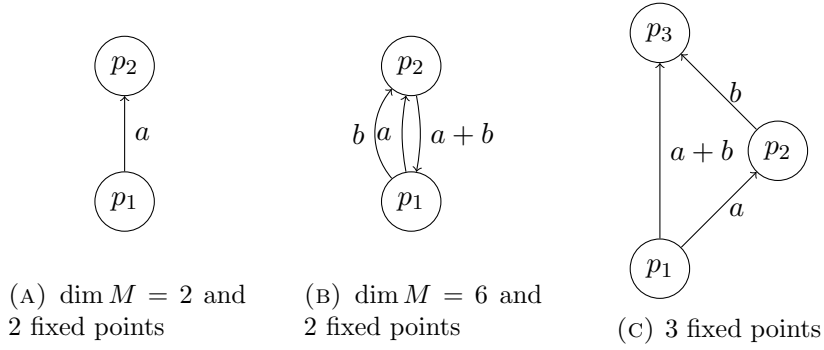
\begin{figure}
\begin{subfigure}[b][5cm][s]{.25\textwidth}
\centering
\vfill
\begin{tikzpicture}[state/.style ={circle, draw}]
\node[state] (a) {$p_1$};
\node[state] (b) [above=of a] {$p_2$};
\path (a) [->]edge node [right] {$a$} (b);
\end{tikzpicture}
\vfill
\caption{$\dim M=2$ and 2 fixed points}\label{fig3-1}
\vspace{\baselineskip}
\end{subfigure}\qquad
\begin{subfigure}[b][5cm][s]{.25\textwidth}
\centering
\vfill
\begin{tikzpicture}[state/.style ={circle, draw}]
\node[state] (a) {$p_1$};
\node[state] (b) [above=of a] {$p_2$};
\path (a) [->] [bend left =10]edge node[left] {$a$} (b);
\path (a) [->] [bend left =40]edge node[left] {$b$} (b);
\path (b) [->] [bend left =10]edge node[right] {$a+b$} (a);
\end{tikzpicture}
\vfill
\caption{$\dim M=6$ and 2 fixed points}\label{fig3-2}
\vspace{\baselineskip}
\end{subfigure}\qquad
\begin{subfigure}[b][5cm][s]{.25\textwidth}
\centering
\vfill
\begin{tikzpicture}[state/.style ={circle, draw}]
\node[state] (a) {$p_1$};
\node[state] (b) [above right=of a] {$p_2$};
\node[state] (c) [above left=of b] {$p_3$};
\path (a) [->] edge node[right] {$a$} (b);
\path (b) [->] edge node[right] {$b$} (c);
\path (a) [->]edge node [right] {$a+b$} (c);
\end{tikzpicture}
\vfill
\caption{3 fixed points}\label{fig3-3}
\vspace{\baselineskip}
\end{subfigure}\qquad
\caption{Multigraphs for Theorem \ref{t21}.}\label{fig3}
\end{figure}

In each case, the corresponding multigraph is given in Figure \ref{fig3}. Let $N^i$ be the number of fixed points with $n_p=i$. Denote by $\chi_y(M)$ the Hirzebruch $\chi_y$-genus of $M$, where the Hirzebruch $\chi_y$-genus is the genus belonging to the power series $\frac{x(1+ye^{-x(1+y)})}{1+e^{-x(1+y)}}$. In \cite{L}, Li proves that the Hirzebruch $\chi_y$-genus is rigid under a circle action having a discrete fixed point set.

\begin{theorem} \label{t22} \cite{L} Let the circle act on a $2n$-dimensional compact almost complex manifold $M$ with isolated fixed points. For each integer $i$ such that $0 \leq i \leq n$,
\begin{center}
$\displaystyle \chi^i(M)=\sum_{p \in M^{S^1}} \frac{\sigma_i (t^{w_p^1}, \cdots, t^{w_p^n})}{\prod_{j=1}^n (1-t^{w_p^j})} = (-1)^i N^i = (-1)^i N^{n-i}$,
\end{center}
where $t$ is an indeterminate, $\sigma_i$ is the $i$-th elementary symmetric polynomial in $n$ variables, and $\chi_y(M)=\sum_{i=0}^n \chi^i(M) \cdot y^i$. \end{theorem}

The Hirzebruch $\chi_y$-genus of a manifold $M$ encodes three invariants of $M$; $\chi_0(M)=\mathrm{Todd}(M)$ is the Todd genus, $\chi_{-1}(M)=\mathrm{sign}(M)$ is the signature, and $\chi_1(M)=\chi(M)$ is the Euler characteristic of $M$. In other words, the Todd genus of $M$ is equal to the number of fixed points whose weights are all positive.

One of important properties of weights is that, given an integer $w$, the number of times the weight $w$ occurs over all fixed points, counted with multiplicity, is the same as the number of times the weight $-w$ occurs over all fixed points, counted with multiplicity.

\begin{lem} \label{l23} \cite{H}, \cite{L} Let the circle act on a compact almost complex manifold $M$ with a discrete fixed point set. Then for any integer $w$,
\begin{center}
$\D{\sum_{p \in M^{S^{1}}} N_{p}(w)=\sum_{p \in M^{S^{1}}} N_{p}(-w)}$,
\end{center}
where $N_{p}(w)$ is the multiplicity of $w$ in the isotropy representation $T_{p} M$. \end{lem}

Moreover, there exist fixed points whose numbers of negative weights differ by 1.

\begin{lemma} \label{l24} \cite{J2} Let the circle act on a compact almost complex manifold $M$ with a non-empty discrete fixed point set such that $\dim M>0$. Then there exists $i$ such that $N^i \neq 0$ and $N^{i+1} \neq 0$. \end{lemma}

Let the circle act effectively on a compact almost complex manifold $M$ with a non-empty discrete fixed point set $M^{S^1}$. Let $w>1$ be a positive integer. As a subgroup of $S^1$, $\mathbb{Z}_w$ acts on $M$. The set $M^{\mathbb{Z}_w}$ of points fixed by the $\mathbb{Z}_w$-action is a union of smaller dimensional almost complex submanifolds. Assume that an $S^1$-fixed point $p \in M^{S^1}$ is contained in a connected component $Z$ of $M^{\mathbb{Z}_w}$. Then $p$ has exactly $m$ weights that are divisible by $w$ if and only if $\dim Z=2m$. We shall denote by $\Sigma_p^Z$ the multiset of weights in the isotropy representation $T_pZ$, $n_p^Z$ the number of negative weights in $T_pZ$, and $N^i_Z$ the number of fixed points in $Z$ with $n_p^Z=i$. If two fixed points $p$ and $p'$ are in the same connected component $Z$, then their weights are equal modulo $w$.

\begin{lem} \label{l25} \cite{T}, \cite{GS} Let the circle act on a compact almost complex manifold $M$. Let $p$ and $p'$ be fixed points which lie in the same connected component of $M^{\mathbb{Z}_{w}}$, for some positive integer $w$. Then the $S^{1}$-weights at $p$ and at $p'$ are equal modulo $w$. \end{lem}

The Lemma \ref{l25} means that, there exists a bijection $\pi:\{1,\cdots,n\} \longrightarrow \{1,\cdots,n\}$ such that $w_p^{i} \equiv w_{p'}^{\pi(i)} \mod w$ for each $i$. 

Let $w>1$ be an integer. Consider an effective circle action on a compact almost complex manifold with a discrete fixed point set.  For each fixed point $p$, denote by $m_p^w$ ($m_p^w(+)$, $m_p^w(-)$, respectively) the number of weights (positive weights, negative weights, respectively) at $p$ that are divisible by $w$. As an immediate consequence of Theorem \ref{t21}, we can prove the following lemma, which will be used frequently, to prove Theorem \ref{t12}.

\begin{lemma} \label{l26} 
Let the circle act effectively on a compact almost complex manifold $M$ with a discrete fixed point set. Let $w>1$ be an integer.
\begin{enumerate}[(1)]
\item If there exists a fixed point $p$ with $m_p^w=1$, then there exists at least one more fixed point $p'$ with $m_{p'}^w=1$.
\item If there exists a fixed point $p$ with $m_p^w=3$, then there exists at least one more fixed point $p'$ with $m_{p'}^w=3$.
\item If there exists a fixed point $p$ with $m_p^w=2$, then there exist at least two more fixed points $p'$ and $p''$ with $m_{p'}^2=2$, $m_{p''}^w=2$. Moreover, one of them satisfies $m_q^w(-)=0$, one of them satisfies $m_q^w(-)=1$, and one of them satisfies $m_q^w(-)=2$.
\item If there exists a fixed point $p$ with $m_p^w \geq 4$, then there exist at least three more fixed points $q$ with $m_q^w=4$.
\end{enumerate}
\end{lemma}

\begin{proof}
Consider the set $M^{\mathbb{Z}_w}$ of points in $M$ that are fixed by the $\mathbb{Z}_w$-action, where $\mathbb{Z}_w$ acts on $M$ as a subgroup of $S^1$. Since the action is effective, $M^{\mathbb{Z}_w}$ is a union of smaller dimensional almost complex manifolds. Let $Z$ be a connected component of $M^{\mathbb{Z}_w}$ that contains an $S^1$-fixed point $p$. Then $p$ has precisely $m$ weights that are divisible by $w$ if and only if $\dim Z=2m$. There is an induced action $S^1=S^1/\mathbb{Z}_w$ on $Z$. The induced action has $p$ as a fixed point. Therefore, by Theorem \ref{t21},
\begin{enumerate}[(1)]
\item if $m=1$, then the induced action on $Z$ must have one more fixed point $p'$.
\item if $m=3$, then the induced action on $Z$ must have at least one more fixed point $p'$.
\item if $m=2$, then the induced action on $Z$ must have at least two more fixed points $p', p''$.
\item if $m \geq 4$, then the induced action on $Z$ must have at least three more fixed points $p'$.
\end{enumerate}
Since $p' \in Z$ (and $p'' \in Z$), $p'$ (and $p''$) has exactly $m$ weights that are divisible by $w$. Suppose that $m=2$. Applying Theorem \ref{t22} for the induced $S^1$-action on $Z$, it follows that $N^0_Z=N^2_Z$. Applying Lemma \ref{l25} for the induced action, there exists $i$ such that $N^i_Z \neq 0$ and $N^{i+1}_Z \neq 0$. It follows that $N^i_Z \neq 0$ for $0 \leq i \leq 2$. This proves the second statement of (3). \end{proof}

The following lemma is a generalization of the result on a semi-free symplectic circle action on a compact symplectic manifold with a discrete fixed point set \cite{TW} \cite{L}. Our proof is adapted from \cite{L}.

\begin{lem} \label{l27} Let the circle act on a $2n$-dimensional compact almost complex manifold with a non-empty discrete fixed point set. Assume that all the weights are $\pm w$ for some positive integer $w$. Then the number of fixed points is $k \cdot 2^n$ for some positive integer $k$. Moreover, $N^i=k \cdot {n \choose i}$, where $N^i$ is the number of fixed points that have exactly $i$ negative weights. \end{lem}

\begin{proof}
Quotient out by the subgroup $\mathbb{Z}_w$ that acts trivially on $M$. Now the action is semi-free, that is, free outside the fixed point set. Therefore, all the weights are now $\pm 1$. By Theorem \ref{t22}, we have
\begin{center}
$\displaystyle \chi^0(M)=N^0=\sum_{p \in M^{S^1}} \frac{1}{\prod_{i=1}^n (1-t^{\pm 1})}=\frac{\sum_{i=0}^n N^i (-t)^i}{(1-t)^n}$.
\end{center}
Since the fixed point set are non-empty, one of $N^i$ is non-zero. This implies that $\chi^0(M)>0$. Assume that $\chi^0(M)=N^0=k$. Then it follows that $N^i=k \cdot {n \choose i}$. \end{proof}

The example of such a manifold is a $k$-copies of $S^2 \times \cdots \times S^2$, where the circle acts on each $S^2$ by rotation at speed $w$. One of the key steps to prove Theorem \ref{t12} is to consider the largest weight among all the weights. The following lemma states how the largest weight behaves. 

\begin{lem} \label{l28}
Let the circle act on a compact almost complex manifold with a discrete fixed point set. Let $l$ be the largest weight. Let $Z$ be a connected component of the set $M^{\mathbb{Z}_l}$ of points fixed by the $\mathbb{Z}_l$-action. If $Z$ contains an $S^1$-fixed point $p$, then the induced $S^1$-action on $Z$ has $k \cdot 2^m$ fixed points for some integer $k$, where $2m=\dim Z$. Moreover, $N_Z^i=k \cdot {m \choose i}$, where $N_{Z}^i$ is the number of fixed points that have the weight $(-l)$ $i$-times. \end{lem}

\begin{proof} Let $Z$ be a connected component of $M^{\mathbb{Z}_l}$. On $Z$, there is an induced action of $S^1=S^1/\mathbb{Z}_l$. If $p$ is a fixed point of the induced action, then the weights in the isotropy representation of $T_pZ$ are $\pm l$. Applying Lemma \ref{l27} for the induced action on $Z$, the lemma follows. \end{proof}

If we push-forward the class $\alpha=1$ in the Atiyah-Bott-Berline-Vergne localization formula, we obtain the following lemma; see, for instance, section 2 of \cite{PT} (take $\alpha=1$ in Theorem 7 of \cite{PT}).

\begin{lem} \label{l29} Let the circle act on a compact almost complex manifold $M$ with a discrete fixed point set such that $\dim M>0$. Then
\begin{center}
$\displaystyle \sum_{p \in M^{S^1}} \frac{1}{\prod_{i=1}^n w_p^i}=0$.
\end{center}
\end{lem}

The following example illustrates a circle action on a 4-dimensional compact almost complex manifold with 4 fixed points, whose weights are as in Theorem \ref{t11}.

\begin{exa} \label{e210}
Let $n$ be an integer. Consider a Hirzebruch surface $\{([z_0:z_1:z_2],[w_1:w_2])\in\mathbb{CP}^2\times\mathbb{CP}^1|z_1 w_2^n=z_2 w_1^n\}$. For $g \in S^1 \subset \mathbb{C}$, let $g$ act by $g \cdot ([z_0:z_1:z_2],[w_1:w_2]) = ([g^a z_0:z_1:g^{nb} z_2],[w_1:g^b w_2])$. The action has 4 fixed points, $p_1=([1:0:0],[1:0]), p_2=([1:0:0],[0:1]), p_3=([0:1:0],[1:0]), p_4=([0:0:1],[0:1])$.
\begin{enumerate}[(1)]
\item At $p_1$, $\displaystyle (\frac{z_1}{z_0},\frac{w_2}{w_1})$ become local coordinates. Locally, the $S^1$-action is given by $\displaystyle g \cdot (\frac{z_1}{z_0},\frac{w_2}{w_1}) = (\frac{z_1}{g^a z_0},\frac{g^b w_2}{w_1})=(g^{-a}\frac{z_1}{z_0},g^b \frac{w_2}{w_1})$. Therefore, the weights at $p_1$ are $\Sigma_{p_1}=\{-a,b\}$.
\item At $p_2$, $\displaystyle(\frac{z_2}{z_0},\frac{w_1}{w_2})$ become local coordinates. Local action of $S^1$-action is given by $\displaystyle g \cdot (\frac{z_2}{z_0},\frac{w_1}{w_2}) = (g^{nb-a} \frac{z_2}{z_0},g^{-b} \frac{w_1}{w_2})$. The weights at $p_2$ are $\Sigma_{p_2}=\{nb-a,-b\}$.
\item At $p_3$, $\displaystyle(\frac{z_0}{z_1},\frac{w_2}{w_1})$ become local coordinates. Local action of $S^1$-action is given by $\displaystyle g \cdot (\frac{z_0}{z_1},\frac{w_2}{w_1}) = (g^{a} \frac{z_0}{z_1},g^{b} \frac{w_2}{w_1})$. The weights at $p_3$ are $\Sigma_{p_3}=\{a,b\}$.
\item At $p_4$, $\displaystyle(\frac{z_0}{z_2},\frac{w_1}{w_2})$ become local coordinates. Local action of $S^1$-action is given by $\displaystyle g \cdot (\frac{z_0}{z_2},\frac{w_1}{w_2}) = (g^{a-nb} \frac{z_0}{z_2},g^{-b} \frac{w_1}{w_2})$. The weights at $p_4$ are $\Sigma_{p_4}=\{a-nb,-b\}$.
\end{enumerate} \end{exa}

The following examples describe a circle action on a 6-dimensional compact almost complex manifold with 4 fixed points, whose weights are the same as the first, second, third (for $a=4,5$ only), and fourth case of Theorem \ref{t12}.

\begin{exa} \label{e211} Let $g \in S^1 \subset \mathbb{C}$ act on $\mathbb{CP}^3$ by $g \cdot [z_0:z_1:z_2:z_3]=[z_0:g^a z_1:g^b z_2:g^c z_3]$ for mutually distinct positive integers $a,b,c$. The fixed points are $[1:0:0:0], [0:1:0:0], [0:0:1:0], [0:0:0:1]$. The weights at the fixed points are $\{a,b,c\},\{-a,b-a,c-a\},\{-b,a-b,c-b\},\{-c,a-c,b-c\}$, respectively. \end{exa}

\begin{exa} \label{e212} \cite{A} Let $g \in S^1 \subset \mathbb{C}$ act on the complex quadric $Q^3=\{[z_0:z_1:z_2:z_3:z_4]\in\mathbb{CP}^4|z_0z_2+z_2z_3+z_4^3=0\}$ by $g \cdot [z_0:z_1:z_2:z_3:z_4]=[g^a z_0:g^{-a} z_1:g^b z_2:g^{-b} z_3:z_4]$ for mutually distinct positive integers $a,b$. The fixed points are $[1:0:0:0:0],[0:1:0:0:0],[0:0:1:0:0],[0:0:0:1:0]$ and the weights at the fixed points are $\{-a,b-a,-b-a\},\{a,b+a,-b+a\},\{-b,a+b,-a-b\},\{b,a+b,-a+b\}$, respectively. \end{exa}

\begin{exa} \label{e213} For the examples of a manifold of the third type in Theorem \ref{t12} with $a=4$ or 5, refer to \cite{A} for complex description and \cite{M} for symplectic description. \end{exa}

\begin{exa} \label{e214} \cite{K2} The 6-sphere $S^6$ can be though of as the quotient of the Lie group $G_2$ by $SU(3)$. It admits an $S^1$-action with two fixed points. By Theorem \ref{t21}, it follows that the weights at the two fixed points are $\Sigma_{p_1}=\{-a-b,a,b\}$ and $\Sigma_{p_2}=\{-a,-b,a+b\}$ for some positive integers $a$ and $b$. By taking a disjoint union of circle actions on two 6-spheres each of which has two fixed points, the fourth type of Theorem \ref{t12} is provided. \end{exa}

\begin{rem} \label{r215} In this remark, we discuss a possibility on the existence of a manifold of the fifth type or the sixth type in Theorem \ref{t12}. Let the circle act on a 6-dimensional compact almost complex manifold with two fixed points (as a rotation on $S^6$) as in Theorem \ref{t21}.

First, suppose that $a<b$ and we are able to blow up (either complex or symplectic) a neighborhood of $p_1$. Then the blow up would replace $p_1$ by three fixed points whose weights are $\Sigma_{q_1}=\{-2a-b,a,b-a\},\Sigma_{q_2}=\{-a-b,2a+b,a+2b\},\Sigma_{q_3}=\{a-b,-a-2b,b\}$. Let $a=A$ and $B=b-a$ and write weights in terms of $A$ and $B$. The blown up manifold $M'$ then would have the fixed point data same as the fifth case of Theorem \ref{t12}.

Second, consider the isotropy sphere $Z=S^2$ whose isotropy group is $\mathbb{Z}_{b}$. Note that the isotropy sphere $Z$ can be realized as the $b$-edge in Figure \ref{fig3-2}. Suppose that we are able to blow up a neighborhood of $Z$. This would result in replacing the two fixed points $p_1$ and $p_2$ by 4 fixed points whose weights are $\{-a-b,a+2b,b\}$, $\{-a-2b,a,b\}$, $\{-a,-b,a+2b\}$, $\{-b,-a-2b,a+b\}$. By changing $a$ and $b$, the sixth case of Theorem \ref{t12} would be achieved. 

Therefore, a problem of the existence of a manifold of the fifrh type and the sixth type may be reduced to the existence problem of a manifold with two fixed points that is (locally) complex or symplectic. While a rotation of $S^6$ provides a manifold with 2 fixed points, it is not known if there is a complex structure preversing circle action on $S^6$ with 2 fixed points. On the other hand, it is not known if there exists a 6-dimensional symplectic manifold with 2 fixed points ($S^6$ cannot be symplectic) and this is an open question in equivariant symplectic geometry. If there exists such a manifold (either complex or symplectic), then the fifth type and the sixth type of Theorem \ref{t12} would be obtained. \end{rem}

\section{Special Multigraphs}

Let the circle act on a compact almost complex manifold $M$ with a discrete fixed point set. It is known that there exists a multigraph that describes $M$ \cite{GS}, \cite{JT}. In this section, we associate a multigraph where we draw edges differently depending on the size of weights; see Lemma \ref{l35}. While our choice looks a bit more complicated, it reduces the computation for the classification of our main results.

First, we discuss the properties of small weights. In \cite{J2}, the author introduces the notion of primitive weights. A positive weight $w$ is called primitive if $w$ cannot be written as the sum of positive weights, other then $w$ itself. Primitive weights are well-behaved in a sense that, for each primitive weight $w$, the number of times the weight $+w$ occurs at fixed points $p$ with $n_p=i$ is equal to the number of times the weight $-w$ occurs at fixed points $p$ with $n_p=i+1$, for all $i$. In this paper, we shall only use the fact that the smallest positive weight and the second smallest positive weight are primitive.

\begin{lem} \label{l31} Let the circle act on a compact almost complex manifold with a discrete fixed point set. Let $w$ be either the smallest positive weight or the second smallest positive weight among weights over all the fixed points, counted with multiplicity. Then for each $i$, the number of times $w$ occurs at fixed points with $n_p=i$ is equal to the number of times $-w$ occurs at fixed points with $n_p=i+1$.  \end{lem}

\begin{rem} The smallest positive weight and the second smallest positive weight may be equal. Moreover, there may be other weights that are equal to the second smallest positive weight. The third smallest positive weight need not be primitive; see (2)(b) of Theorem \ref{t21} or Example \ref{e213}. \end{rem}

Let the circle act on a compact almost complex manifold $M$ with a discrete fixed point set. Let $w$ be a positive integer. As we have seen in Lemma \ref{l23}, the number of times weight $w$ occurs over all the fixed points, counted with multiplicity, is the same as the number of times weight $-w$ occurs over all the fixed points, counted with multiplicity. From this, a multigraph associated to $M$ has been considered either implicitly or explicitly; we draw an edge $\epsilon$ from a fixed point $p$ having weight $w$ to a fixed point $p'$ having weight $-w$. The direction implies that an edge goes from a fixed point $p$ having positive weight $w$ to a fixed point $p'$ having negative weight $-w$. We label the edge by $w$ to encode weight $w$.

\begin{definition} \label{d32} A \textbf{labelled, directed multigraph} consists of a set $V$ of vertices, a set $E$ of edges, maps $i \colon E \to V$ and $t \colon E \to V$ that give the initial vertex and the terminal vertex of each edge, and a map $w \colon E \to \mathbb{N}^+$ where $\mathbb{N}^+$ is the set of positive integers. \end{definition}

\begin{definition} \label{d33} Let the circle act on a compact almost complex manifold $M$ with a discrete fixed point set. A (labelled, directed) multigraph is called a \textbf{multigraph associated to $M$} if for any fixed point $p$, the multiset of weights at $p$ are $\{w(\epsilon) \mid i(\epsilon)=p\} \cup \{-w(\epsilon) \mid t(\epsilon)=p\}$. \end{definition}

An edge $\epsilon$ is called a \textbf{loop} if $i(\epsilon)=t(\epsilon)$. We show that we can associate a multigraph without any loop that satisfies extra properties.

\begin{lem} \label{l35}
Let the circle act on a compact almost complex manifold $M$ with a discrete fixed point set. Then there exists a multigraph associated to $M$ with following properties:
\begin{enumerate}[(1)]
\item Given an edge $\epsilon$, if $w(\epsilon)$ is smaller than or equal to the second smallest positive weight, then $n_{i(\epsilon)}+1=n_{t(\epsilon)}$.
\item Given an edge $\epsilon$, if $w(\epsilon)$ is strictly bigger than the second smallest positive weight, then the weights at $i(\epsilon)$ and the weights at $t(\epsilon)$ are equal modulo $w(\epsilon)$.
\item The graph has no loops. \end{enumerate} \end{lem}

\begin{proof}
Let $w$ be a positive weight that is smaller than or equal to the second smallest positive weight. By Lemma \ref{l31}, for each $i$, the number of times $w$ occurs at fixed points with $n_p=i$ is equal to the number of times $-w$ occurs at fixed points with $n_p=i+1$. We draw edges $\epsilon$ with label (weight) $w(\epsilon)=w$ from fixed points with $n_p=i$ having weight $w$ to fixed points with $n_p=i+1$ having weight $-w$. This proves the first part.

Let $w$ be an integer that is strictly bigger than the second smallest positive weight. The smallest positive weight in the isotropy submanifold $M^{\mathbb{Z}/w}$ is $w$ itself. Let $Z$ be a connected component of $M^{\mathbb{Z}/w}$. Applying Lemma \ref{l31} for the induced action of $S^1=S^1/\mathbb{Z}_w$, the number of times the weight $+w$ occurs in points with $n_p^Z=i$ in $Z$ is equal to the number of times the weight $-w$ occurs in points with $n_p^Z=i+1$ in $Z$, for each $i$. Draw edges $\epsilon$ whose label is $w$ accordingly by this recipe. Let $\epsilon$ be an edge with $w(\epsilon)=w$. Then by our choice of graph, $i(\epsilon)$ and $t(\epsilon)$ are in the same connected component $Z$ of $M^{\mathbb{Z}/w}$. Therefore, by Lemma \ref{l25}, the weights at $i(\epsilon)$ and the weights at $t(\epsilon)$ are equal modulo $w$. \end{proof}

\begin{rem} In \cite{JT}, a multigraph is said to \emph{describe} $M$ if in addition, the two endpoints $i(\epsilon)$ and $t(\epsilon)$ are in the same component of the isotropy submanifold $M^{\mathbb{Z}/(w(\epsilon))}$ for each edge $\epsilon$. It is called an \emph{integral multigraph} in \cite{GS}. In this paper, we draw an edge $\epsilon$ as in \cite{JT} if $w(\epsilon)$ is not small. If a weight is small, we want to draw an edge differently, by using Lemma \ref{l31}. While our choice of multigraphs seems more complicated, it reduces the number of multigraphs to consider. \end{rem}

\section{4-dimension with 4 fixed points}

In this section, we prove Theorem \ref{t11}; we determine the fixed point data of a circle action on a compact almost complex manifold $M$, when the dimension of the manifold is 4 and there are 4 fixed points. The proof of Theorem \ref{t11} will be given at the end of this section. First, we consider a multigraph associated to $M$, that satisfies the conditions in Lemma \ref{l35}.

\begin{lem} \label{l41}
Let the circle act on a 4-dimensional compact almost complex manifold with 4 fixed points. Then $N^0=1$, $N^2=2$, and $N^2=1$. Let $p_i$ be fixed points such that $n_{p_1}=0$, $n_{p_2}=1$, $n_{p_3}=1$, and $n_{p_4}=2$. Then exactly one of the figures in Figure \ref{fig4} occurs as a multigraph associated to $M$ that satisfies the conditions in Lemma \ref{l35}. Alternatively, exactly one of the following holds for the multisets of weights at $p_i$:
\begin{enumerate}[(1)]
\item $\Sigma_{p_1}=\{a,b\}$, $\Sigma_{p_2}=\{-a,c\}$, $\Sigma_{p_3}=\{-b,d\}$, and $\Sigma_{p_4}=\{-c,-d\}$ for some positive integers $a,b,c$, and $d$.
\item $\Sigma_{p_1}=\{a,b\}$, $\Sigma_{p_2}=\{-a,c\}$, $\Sigma_{p_3}=\{-c,d\}$, and $\Sigma_{p_4}=\{-b,-d\}$ for some positive integers $a,b,c$, and $d$.
\end{enumerate}
\end{lem}

\begin{proof} By Theorem \ref{t22}, $N^0=N^2$. By Lemma \ref{l24}, there exists $i$ such that $N^i\neq0$ and $N^{i+1}\neq0$. These imply that $N^0=N^2=1$ and $N^1=2$.

We draw a multigraph in the sense of Lemma \ref{l35}. Label fixed points by $p_i$ such that $n_{p_1}=0$, $n_{p_2}=1$, $n_{p_3}=1$, and $n_{p_4}=2$. Let $a$ and $b$ be positive weights at $p_1$. By the first condition of Lemma \ref{l35}, there cannot be two edges between $p_1$ and $p_4$. By condering possible multigraphs, the lemma follows. \end{proof}

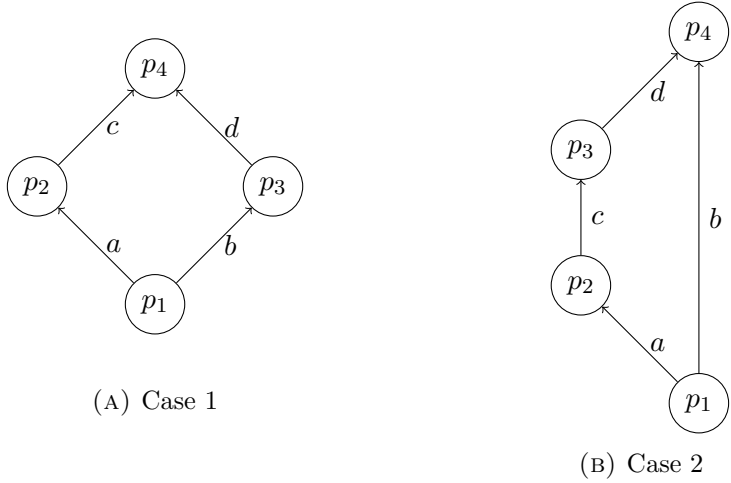
\begin{figure}
\centering
\begin{subfigure}[b][6cm][s]{.45\textwidth}
\centering
\vfill
\begin{tikzpicture}[state/.style ={circle, draw}]
\node[state] (A) {$p_1$};
\node[state] (B) [above left=of A] {$p_2$};
\node[state] (C) [above right=of A] {$p_3$};
\node[state] (D) [above right=of B] {$p_4$};
\path (A) [->] edge node[right] {$a$} (B);
\path (A) [->] edge node [right] {$b$} (C);
\path (B) [->] edge node [right] {$c$} (D);
\path (C) [->] edge node [right] {$d$} (D);
\end{tikzpicture}
\vfill
\caption{Case 1}\label{fig4-1}
\vspace{\baselineskip}
\end{subfigure}\qquad
\begin{subfigure}[b][6cm][s]{.45\textwidth}
\centering
\vfill
\begin{tikzpicture}[state/.style ={circle, draw}]
\node[state] (A) {$p_1$};
\node[state] (B) [above left=of A] {$p_2$};
\node[state] (C) [above =of B] {$p_3$};
\node[state] (D) [above right=of C] {$p_4$};
\path (A) [->] edge node[right] {$a$} (B);
\path (A) [->] edge node [right] {$b$} (D);
\path (B) [->] edge node [right] {$c$} (C);
\path (C) [->] edge node [right] {$d$} (D);
\end{tikzpicture}
\vfill
\caption{Case 2}\label{fig4-2}
\end{subfigure}
\caption{Multigraphs for 4-dimension with 4 fixed points.}\label{fig4}
\end{figure}

Therefore, our task to prove Theorem \ref{t11} is to classify the multisets of weights at the fixed points in each case of Lemma \ref{l41}.

\begin{lem} \label{l42} Suppose that the first case in Lemma \ref{l41} holds. Then either $b=c$ or $a=d$. If $b=c$, then either $a \equiv d \mod b$, or $a \equiv -d \mod b$. If $a=d$, then either $b \equiv c \mod a$, or $b \equiv -c \mod a$. \end{lem}

\begin{proof} By Lemma \ref{l29}, we have
\begin{center}
$\displaystyle 0=\frac{1}{ab}+\frac{1}{(-a)c}+\frac{1}{(-b)d}+\frac{1}{(-c)(-d)}=(\frac{1}{a}-\frac{1}{d})(\frac{1}{b}-\frac{1}{c})$.
\end{center}
Therefore, either $b=c$ or $a=d$.

First, suppose that $b=c$. By quotienting out by the subgroup that acts trivially, we may assume that the action is effective. The lemma follows if $b=1$. Next, suppose that $b>1$. Consider $M^{\mathbb{Z}_b}$, the set of points in $M$ that are fixed by the $\mathbb{Z}_b$-action, where $\mathbb{Z}_b$ acts on $M$ as a subgroup of $S^1$. Let $Z$ be a connected component of $M^{\mathbb{Z}_b}$ that contains $p_1$. Since the action is effective, $\dim Z=2$. Moreover, there is an induced action of $S^1=S^1/\mathbb{Z}_b$ that acts on $Z$. Since the induced action has $p_1$ as a fixed point, it follows that $Z$ is the 2-sphere and it contains another fixed point $p$ that has weight $-b$. Therefore, $p$ is either $p_3$ or $p_4$. By Lemma \ref{l25}, the weights at $p_1$ and the weights at $p$ are equal modulo $b$. If $p$ is $p_3$ we have $a \equiv d \mod b$ and if $p$ is $p_4$ we have $a \equiv -d \mod b$.

By the symmetry between $a$ and $b$, and $c$ and $d$, the other case follows. \end{proof}

\begin{lem} \label{l43} Suppose that the second case in Lemma \ref{l41} holds. Then $b=c$. Moreover, either $a \equiv d \mod b$ or $a \equiv -d \mod b$. \end{lem}

\begin{proof} By Lemma \ref{l29}, we have
\begin{center}
$\displaystyle 0=\frac{1}{ab}+\frac{1}{(-a)c}+\frac{1}{(-c)d}+\frac{1}{(-b)(-d)}=(\frac{1}{a}+\frac{1}{d})(\frac{1}{b}-\frac{1}{c})$.
\end{center}
Therefore, we have that $b=c$. For the rest we proceed as in the proof of Lemma \ref{l42}. Let $Z$ be a connected component of $M^{\mathbb{Z}_b}$ that contains $p_1$ that has weight $b$. If $Z$ contains $p_3$ then $a \equiv d \mod b$ by Lemma \ref{l25}. If $Z$ contains $p_4$ then $a \equiv -d \mod b$.\end{proof}

With the above, we are ready to prove Theorem \ref{t11}.

\begin{proof} [Proof of Theorem \ref{t11}] This follows from Lemma \ref{l41}, Lemma \ref{l42}, and Lemma \ref{l43}. \end{proof}

\begin{rem} Given a circle action on a compact almost complex manifold $M$ with a discrete fixed point set, one may wonder if Theorem \ref{t22} gives all the information on the fixed point data, as it does, for instance for semi-free symplectic circle actions (Theorem 3.3 of \cite{L}) or when there are two fixed points (Theorem 2.8 of \cite{J3}). When $\dim M=4$ and there are 4 fixed points, from Lemma \ref{l41} one can check that Theorem \ref{t22} tells us the fixed point data is $\{a,b\},\{-a,b\},\{-b,c\},\{-b,-c\}$ for some positive integers $a,b,c$. However, Theorem \ref{t22} does not tell us that we must have either $a \equiv c \mod b$ or $a \equiv -c \mod b$. \end{rem}

\section{6-dimension with 4 fixed points: the case that $\mathrm{Todd}(M)=1$}

In this section, we classify the fixed point data of a circle action on a 6-dimensional almost complex manifold $M$ with 4 fixed points, when the Todd genus of $M$ is 1. First, we consider a multigraph associated to $M$ in the sense of Lemma \ref{l35}. And then in each case, we classify the fixed point data.

\begin{lem} \label{l51}
Let the circle act effectively on a 6-dimensional compact almost complex manifold with 4 fixed points, whose Todd genus is 1. Then $N^i=1$ for $i=0,1,2,3$. Let $p_i$ be a fixed point with $n_{p_i}=i$, for $i=0,1,2,3$. Then there exist positive integers $a,b,c,d,e,f$ so that exactly one of the figures in Figure \ref{fig5} occurs as a multigraph associated to $M$ that satisfies the conditions in Lemma \ref{l35}. Alternatively, exactly one of the following holds for the multisets of weights at $p_i$:
\begin{enumerate}[(1)]
\item $\Sigma_{p_0}=\{a,b,c\}, \Sigma_{p_1}=\{-d,e,f\}, \Sigma_{p_2}=\{-e,-f,d\}, \Sigma_{p_3}=\{-a,-b,-c\}$.
\item $\Sigma_{p_0}=\{a,b,c\}, \Sigma_{p_1}=\{-c,d,e\}, \Sigma_{p_2}=\{-d,-e,f\}, \Sigma_{p_3}=\{-a,-b,-f\}$.
\item $\Sigma_{p_0}=\{a,b,c\}, \Sigma_{p_1}=\{-a,d,e\}, \Sigma_{p_2}=\{-b,-d,f\}, \Sigma_{p_3}=\{-c,-e,-f\}$.
\item $\Sigma_{p_0}=\{a,b,c\}, \Sigma_{p_1}=\{-c,d,e\}, \Sigma_{p_2}=\{-a,-b,f\}, \Sigma_{p_3}=\{-d,-e,-f\}$.
\end{enumerate}
\end{lem}

\begin{figure}
\centering
\begin{subfigure}[b][7cm][s]{.2\textwidth}
\centering
\vfill
\begin{tikzpicture}[state/.style ={circle, draw}]
\node[state] (a) {$p_0$};
\node[state] (b) [above=of a] {$p_1$};
\node[state] (c) [above=of b] {$p_2$};
\node[state] (d) [above=of c] {$p_3$};
\path (a) [->] [bend left =60] edge node[right] {$a$} (d);
\path (a) [->] [bend left =40] edge node [right] {$b$} (d);
\path (a) [->] [bend left =25] edge node [right] {$c$} (d);
\path (b) [->] [bend right =30] edge node [right] {$f$} (c);
\path (b) [->] edge node [right] {$e$} (c);
\path (c) [->] [bend right =30] edge node [right] {$d$} (b);
\end{tikzpicture}
\vfill
\caption{Case 1}\label{fig5-1}
\vspace{\baselineskip}
\end{subfigure}\qquad
\begin{subfigure}[b][7cm][s]{.2\textwidth}
\centering
\vfill
\begin{tikzpicture}[state/.style ={circle, draw}]
\node[state] (a) {$p_0$};
\node[state] (b) [above=of a] {$p_1$};
\node[state] (c) [above=of b] {$p_2$};
\node[state] (d) [above=of c] {$p_3$};
\path (a) [->] [bend left =40] edge node[right] {$a$} (d);
\path (a) [->] [bend left =20] edge node [right] {$b$} (d);
\path (a) [->] edge node [right] {$c$} (b);
\path (b) [->] [bend left=20] edge node [right] {$d$} (c);
\path (b) [->] [bend right =20]  edge node [right] {$e$} (c);
\path (c) [->] edge node [right] {$f$} (d);
\end{tikzpicture}
\vfill
\caption{Case 2}\label{fig5-2}
\vspace{\baselineskip}
\end{subfigure}\qquad
\begin{subfigure}[b][7cm][s]{.2\textwidth}
\centering
\vfill
\begin{tikzpicture}[state/.style ={circle, draw}]
\node[state] (a) {$p_0$};
\node[state] (b) [above=of a] {$p_1$};
\node[state] (c) [above=of b] {$p_2$};
\node[state] (d) [above=of c] {$p_3$};
\path (a) [->] edge node[right] {$a$} (b);
\path (a) [->] [bend left =40] edge node [right] {$b$} (c);
\path (a) [->] [bend left =40]edge node [right] {$c$} (d);
\path (b) [->] edge node [right] {$d$} (c);
\path (b) [->] [bend right =30]  edge node [right] {$e$} (d);
\path (c) [->] edge node [right] {$f$} (d);
\end{tikzpicture}
\vfill
\caption{Case 3}\label{fig5-3}
\end{subfigure}
\begin{subfigure}[b][7cm][s]{.2\textwidth}
\centering
\vfill
\begin{tikzpicture}[state/.style ={circle, draw}]
\node[state] (a) {$p_0$};
\node[state] (b) [above=of a] {$p_1$};
\node[state] (c) [above=of b] {$p_2$};
\node[state] (d) [above=of c] {$p_3$};
\path (a) [->] [bend left =60] edge node[right] {$a$} (c);
\path (a) [->] [bend left =35] edge node [right] {$b$} (c);
\path (a) [->] edge node [right] {$c$} (b);
\path (b) [->] [bend right=60] edge node [right] {$e$} (d);
\path (b) [->] [bend right =30]  edge node [right] {$d$} (d);
\path (c) [->] edge node [right] {$f$} (d);
\end{tikzpicture}
\vfill
\caption{Case 4}\label{fig5-4}
\vspace{\baselineskip}
\end{subfigure}\qquad
\caption{6 dimension, 4 fixed points, $\mathrm{Todd}(M)=1$.}\label{fig5}
\end{figure}
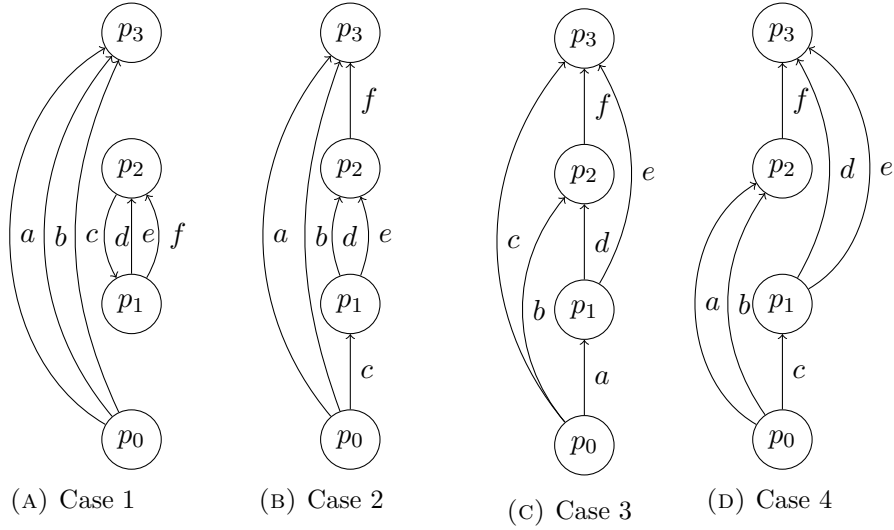

\begin{proof}
By Theorem \ref{t22}, $1=\mathrm{Todd}(M)=N^0=N^3$ and $N^1=N^2$. Hence, $N^i=1$ for $0 \leq i \leq 3$. Let $p_i$ be fixed points with $n_{p_i}=i$, for $0 \leq i \leq 3$. There exists a multigraph associated to $M$ that satisfies properties in Lemma \ref{l35}. Let $m_i$ be the number of edges between $p_0$ and $p_i$, for $i=1,2,3$. Then we have $m_1+m_2+m_3=3$, $0 \leq m_i$ for all $i$, $m_1 \leq 1$, and $m_2 \leq 2$.

First, suppose that $m_1=0$. If $m_2>0$, then any multigraph associated to $M$ without any loop does not have at least two edges $\epsilon$ with $n_{i(\epsilon)}+1=n_{t(\epsilon)}$, which contradicts the first property of Lemma \ref{l35}. Therefore, if $m_1=0$, then $m_2=0$ and $m_3=3$. This is the first case of the lemma (Figure \ref{fig5-1}).

Second, suppose that $m_1=1$. Then either $m_2=0$, $m_2=1$, or $m_2=2$. This is the second (Figure \ref{fig5-2}), third (Figure \ref{fig5-3}), and fourth case (Figure \ref{fig5-4}) of the lemma, respectively. \end{proof}

Let $l$ be the largest weight that occurs. In the next lemma, we prove that the maximum dimension of $M^{\mathbb{Z}_l}$ is 2. In other words, at each fixed point, there cannot be more than one weight that is divisible by $l$.

\begin{lem} \label{l52}
Let the circle act effectively on a 6-dimensional compact almost complex manifold with 4 fixed points and with $\mathrm{Todd}(M)=1$. Let $l$ be the largest weight among all the weights. Then $l>1$ and the dimension of a connected component of $M^{\mathbb{Z}_l}$ that contains a fixed point having the weight $l$ is two. In particular, any fixed point cannot have more than one weight that is a multiple of $l$.
\end{lem}

\begin{proof}
We can associate a multigraph that satisfies the conditions in Lemma \ref{l35}. By Lemma \ref{l51}, one of the figures in Figure \ref{fig7} occurs as a multigraph associated to $M$. Since there exists an edge $\epsilon$ with $n_{i(\epsilon)}+1 \neq n_{t(\epsilon)}$, it follows that $l>1$. Suppose that a fixed point $p$ has the weight $l$. Consider the set $M^{\mathbb{Z}_l}$ of points fixed by the $\mathbb{Z}_l$-action. Let $Z$ be a connected component of $M^{\mathbb{Z}_l}$ that contains $p$. Since $l>1$ and the action is effective, $2 \leq \dim Z \leq 4$. Suppose that $\dim Z=4$. By Lemma \ref{l28}, $Z$ contains $k \cdot 2^2$ fixed points for some positive integer $k$, i.e., $Z$ contains all the 4 fixed points. Moreover, $N_Z^i={2 \choose i}$ for $0 \leq i \leq 2$. It follows that the weights at the fixed points are then $\{l,l,a\}, \{-l,l,b\}, \{-l,-c,l\}, \{-d,-l,-l\}$ for some positive integers $a,b,c,d<l$. By Lemma \ref{l25}, the weights at any two fixed points are equal modulo $l$, since they all lie in $Z$. This implies that $a=b=l-c=l-d$, i.e., the weights are $\{l,l,a\}, \{-l,l,a\}, \{-l,a-l,l\}, \{a-l,-l,-l\}$. By Lemma \ref{l23}, $a=l-a$, i.e., $l=2a$. By the effectiveness of the action, this implies that $a=1$. The weights are then $\{2,2,1\}, \{-2,2,1\}, \{-1,-2,2\}, \{-1,-2,-2\}$. By Theorem \ref{t22},
\begin{center}
$\displaystyle 1=\chi^0(M)=\frac{1}{(1-t^2)^2(1-t)}+\frac{1}{(1-t^{-2})(1-t^2)(1-t)}+\frac{1}{(1-t^{-1})(1-t^{-2})(1-t^2)}+\frac{1}{(1-t^{-1})(1-t^{-2})(1-t^{-2})}=\frac{1-t^2+t^3-t^5}{(1-t)(1-t^2)^2}$,
\end{center}
which cannot hold. Therefore, $\dim Z=2$. \end{proof}

\begin{rem} \label{r53} To reduce the proof, we may reverse the circle action. By reversing a circle action we mean that $g\in S^1$ acts on $M$ by $g \cdot p = g^{-1} p$ for every $p \in M$. At each fixed point, this reverses the sign of each weight. For instance, suppose that we consider Case 2 in Lemma \ref{l51}. We begin by which weight is the largest weight. Suppose that $c$ is the largest weight. Then reverse the circle action. Then the weights are $\Sigma_{p_3}=\{a,b,f\},\Sigma_{p_2}=\{-f,d,e\},\Sigma_{p_1}=\{-d,-e,c\},\Sigma_{p_0}=\{-a,-b,-c\}$. The associated multigraph is the same, with $p_0$ and $p_3$ changed, $p_1$ and $p_2$ changed, and $c$ and $f$ changed. Therefore, once we deal with the case where $c$ is the largest weight, then we do not need to deal with the case where $f$ is the largest weight. We frequently use this phenomenon in the proof of Theorem \ref{t12}. \end{rem}

\begin{lem} \label{l53}
The first case in Lemma \ref{l51} does not hold.
\end{lem}

\begin{proof} The weights $e$ and $f$ are the only weights whose corresponding edge $\epsilon$ satisfies $n_{i(\epsilon)}+1=n_{t(\epsilon)}$. By Lemma \ref{l35}, this implies that $\{e,f\}$ are the first and second smallest positive weights and other positive weights are strictly bigger then $e,f$. In particular, $e<d$ and $f<d$. Therefore, by the second condition of Lemma \ref{l35} for the edge $\epsilon_d$ whose label (weight) is $d$, the weights at $p_1$ and $p_2$ are equal modulo $d$, i.e., $\{e,f,-d\} \equiv \{-e,-f,d\} \mod d$. These imply that $e+f=d$. By Theorem \ref{t22}, we have
\begin{center}
$\displaystyle \chi^0(M)=N^0=1=\frac{1}{(1-t^a)(1-t^b)(1-t^c)}+\frac{1}{(1-t^{-e-f})(1-t^e)(1-t^f)}+\frac{1}{(1-t^{e+f})(1-t^{-e})(1-t^{-f})}+\frac{1}{(1-t^{-a})(1-t^{-b})(1-t^{-c})}=\frac{1}{(1-t^a)(1-t^b)(1-t^c)}-\frac{t^{e+f}}{(1-t^{e+f})(1-t^e)(1-t^f)}+\frac{t^{e+f}}{(1-t^{e+f})(1-t^{e})(1-t^{f})}-\frac{t^{a+b+c}}{(1-t^{a})(1-t^{b})(1-t^{c})}=\frac{1-t^{a+b+c}}{(1-t^a)(1-t^b)(1-t^c)}$, 
\end{center}
which is impossible. \end{proof}

\begin{lem} \label{l54} Suppose that the second case in Lemma \ref{l51} holds. Then the multisets of weights are $\{1,2,3\}$, $\{-1,1,a\}$, $\{-1,-a,1\}$, and $\{-1,-2,-3\}$ for some positive integer $a$. This is the third case of Theorem \ref{t12}. \end{lem}

\begin{proof} By the symmetry between $a$ and $b$, $d$ and $e$, and by reversing the circle action (see Remark \ref{r53}), we may assume that one of the following holds for the largest weight:
\begin{enumerate}
\item $c$ is the largest weight.
\item $a$ is the largest weight.
\item $d$ is the largest weight.
\end{enumerate}

First, suppose that the case (1) holds. Since $\Sigma_{p_0}=\{a,b,c\}, \Sigma_{p_1}=\{-c,d,e\}$, and $c$ is the largest weight, by Lemma \ref{l52}, it follows that $a,b,d,e<c$. By the second condition of Lemma \ref{l35}, $\Sigma_{p_0}=\{a,b,c\} \equiv \{-c,d,e\}=\Sigma_{p_1} \mod c$. This implies that $\{a,b\}=\{d,e\}$. By the first condition of Lemma \ref{l35}, $a$ and $b$ are strictly biggest than the second smallest positive weight. This implies that $d$ and $e$ are also bigger than the second smallest positive weight. It follows that $f$ is the smallest positive weight. However, there is no second smallest positive weight, which is a contradiction. Therefore, $c$ cannot be the largest weight.

Second, suppose that the case (2) holds. By Lemma \ref{l52}, $b,c,f<a$. By the second condition of Lemma \ref{l35}, $\Sigma_{p_0}=\{a,b,c\} \equiv \{-a,-b,-f\}=\Sigma_{p_3} \mod a$. Then either
\begin{enumerate}[(a)]
\item $2b=a$ and $c+f=a$, or
\item $b+f=a$ and $b+c=a$.
\end{enumerate}

Assume that the case (a) holds. If $b=c$, then the weights at $p_0$ are $\{a,b,c\}=\{2b,b,b\}$. The effectiveness of the action implies that $b=1$. However, by the first condition of Lemma \ref{l35}, $b$ cannot be the smallest positive weight since its edge $\epsilon$ has $n_{i(\epsilon)}+3=n_{t(\epsilon)}$. Therefore, $b \neq c$. Since $c+f=a=2b$, this also implies that $f \neq b$. Since $p_0$ has $m_{p_0}^b=2$, by (3) of Lemma \ref{l26}, $p_1$ or $p_2$ must have two weights that are divisible by $b$. Since $c$ and $f$ are not multiples of $b$, $d$ and $e$ must be multiples of $b$. Then we have $m_{p_0}^b(-)=0$, $m_{p_1}^b(-)=0$, $m_{p_2}^b(-)=2$, and $m_{p_3}^b(-)=2$, which contradicts the second statement of (3) of Lemma \ref{l26}.

Assume that the case (b) holds. Similar to the case (a), we have that $b \neq c$. Suppose that $b>c$. Then by the second condition of Lemma \ref{l35} for $b$, $\Sigma_{p_0}=\{b+c,b,c\}\equiv\{-b-c,-b,-c\}=\Sigma_{p_3} \mod b$. This implies that $b=2c$. Then it follows that $a=3c$ and $\Sigma_{p_0}=\{3c,2c,c\}$. Since the action is effective, this implies that $c=1$. We show that at least one of $d$ or $e$ is equal to 1. For this, assume not. Since the biggest weight is $a=3$, $\{d,e\}=\{2,2\}, \{2,3\}$, or $\{3,3\}$. If it is $\{2,2\}$ or $\{3,3\}$, then by (3) of Lemma \ref{l26}, $p_0$ or $p_3$ must have precisely two weights that are multiples of 2 or 3, respectively. However, the weights at $p_0$ and $p_3$ are $\{1,2,3\}$ and $\{-1,-2,-3\}$, which is a contradiction. Assume that $\{d,e\}=\{2,3\}$. Without loss of generality, let $e=3$. Then by the second condition of Lemma \ref{l35} for $e=3$, we have $\Sigma_{p_1}=\{-1,2,3\} \equiv \{-2,-3,1\}=\Sigma_{p_2} \mod 3$. However, this cannot hold. Therefore, at least one of $d$ or $e$ is equal to 1. This is the third case of Theorem \ref{t12}.

Next, suppose that $c>b$. Since $c=f$ and $b<c$, by Lemma \ref{l35} it implies that $\{d,e\}$ must be the smallest and the second smallest positive weights. By the second condition of Lemma \ref{l35} for $c$, $\Sigma_{p_0}=\{b+c,b,c\}\equiv\{-c,d,e\}=\Sigma_{p_1} \mod c$, but this cannot hold.


Third, suppose that the case (3) holds. By Lemma \ref{l52}, $c,e,f<d$. By the second condition of Lemma \ref{l35}, $\Sigma_{p_1}=\{-c,d,e\} \equiv \{-d,-e,f\}=\Sigma_{p_2} \mod d$. This implies that either
\begin{enumerate}[(a)]
\item $c=e$ and $e=f$, or
\item $c+f=d$ and $2e=d$.
\end{enumerate}

Assume that the case (a) holds, i.e., $c=e=f$. Then $c$ is either the smallest positive weight or the second smallest positive weight. Suppose that $c>1$. Then we have $m_{p_i}^c=2$ for $i=1,2$. By (3) of Lemma \ref{l26}, it follows that $m_{p_0}^c=2$ or $m_{p_3}^c=2$. This implies that exactly one of $a$ and $b$ is a multiple of $c$. Without loss of generality, let $a=kc$ for some positive integer $k$. If we apply Lemma \ref{l29} to the induced action of $S^1=S^1/\mathbb{Z}_c$ on $M^{\mathbb{Z}_c}$, we have
\begin{center}
$\displaystyle 0=\sum_{p \in (M^{\mathbb{Z}_c})^{S^1}}\frac{1}{\prod w_p^i}=\frac{1}{kc^2}-\frac{1}{c^2}-\frac{1}{c^2}+\frac{1}{kc^2}$.
\end{center}
Therefore, $k=1$ and $a=c$. However, this contradicts the first condition in Lemma \ref{l35} since $c$ is either the smallest or the second smallest positive weight, but the edge $\epsilon$ for the weight $a$ has $n_{i(\epsilon)}+3=n_{t(\epsilon)}$. Therefore, $c=e=f=1$.

Next, we show that $a \neq b$. For this, suppose that $a=b$. Then $m_{p_0}^b=2$ and hence by (3) of Lemma \ref{l26}, we must have $m_{p_1}^a=2$ or $m_{p_2}^a=2$. However, since $c=e=f=1$, $p_2$ and $p_3$ cannot have two weights that are divisible by $a$. Therefore, $a \neq b$. Without loss of generality, let $a>b$. By the second condition of Lemma \ref{l35} for $a$, $\Sigma_{p_0}=\{a,b,1\} \equiv \{-a,-b,-1\}=\Sigma_{p_3} \mod a$. Note that $a>b>1$ and hence $a>2$. This implies that $a=b+1$. Next, By the second condition of Lemma \ref{l35} for $b$, we have $\Sigma_{p_0}=\{b+1,b,1\} \equiv \{-b-1,-b,-1\}=\Sigma_{p_3} \mod b$. It follows that $b=2$. This is the third case of Theorem \ref{t12}.

Assume that the case (b) holds. If $c=f$, then the weights at $p_1$ are $\{-c,2c,c\}$. Since the action is effective, this implies that $c=1$. By the second condition of Lemma \ref{l35}, this implies that $a=b=2$, since the biggest weight is 2. Then we have $m_{p_0}^2=2$, $m_{p_1}^2=1$, $m_{p_2}^2=1$, and $m_{p_3}^2=2$, which contradicts (3) of Lemma \ref{l26}.

It follows that $c \neq f$. Without loss of generality, by reversing the circle action (see Remark \ref{r53}), we may assume that $c>f$. By the first condition of Lemma \ref{l35}, two of $c,d,e,f$ are the smallest and the second smallest positive weights. On the other hand, $d$ is the biggest weight, $d=2e$, $d=c+f$, and $c>f$. This implies that $f$ is the smallest positive weight and $e$ is the second smallest positive weight. The second condition of Lemma \ref{l35} also implies that $a>e=\frac{d}{2}$ and $b>e=\frac{d}{2}$.

Suppose that $a=b$. Then we have $m_{p_0}^a=2$. By (3) of Lemma \ref{l26}, this implies that $m_{p_1}^a=2$ or $m_{p_2}^a=2$. It follows that $d$ and $e$ must be divisible by $a$, which is a contradiction since $a>e$.

Therefore, $a \neq b$. Without loss of generality, assume that $a>b$. By the second condition of Lemma \ref{l35} for $a$, we have $\Sigma_{p_0}=\{a,b,c\}\equiv\{-a,-b,-f\}=\Sigma_{p_3} \mod a$. Then one of the following holds:
\begin{enumerate}[(i)]
\item $2b=a$ and $c+f=a$.
\item $b+f=a$ and $b+c=a$.
\end{enumerate}
The case (i) is impossible since $a\leq d$ and $\frac{d}{2}<b$. The case (ii) is also impossible since $c \neq f$. \end{proof}

\begin{lem} \label{l55} Suppose that the third case in Lemma \ref{l51} holds. Then the multisets of weights are either
\begin{enumerate}[(1)]
\item $\{a,b,c\},\{-a,b-a,c-a\},\{-b,a-b,c-b\},\{-c,a-c,b-c\}$ for some positive integers $a,b,c$, or
\item $\{a,a+b,a+2b\}$, $\{-a,b,a+2b\}$, $\{-a-2b,-b,a\}$, $\{-a-2b,-a-b,-a\}$ for some positive integers $a,b$.
\end{enumerate}
This is the first case or the second case of Theorem \ref{t12}. \end{lem}

\begin{proof}
By reversing the circle action if necessary (Remark \ref{r53}), we may assume that one of the following holds for the largest weight:
\begin{enumerate}
\item $a$ is the biggest weight.
\item $b$ is the biggest weight.
\item $c$ is the biggest weight.
\item $d$ is the biggest weight.
\end{enumerate}

First, suppose that the case (1) holds. By Lemma \ref{l52}, $b,c,d,e<a$. Since $a,d,f$ are the only weights whose edges $\epsilon$ satisfy $n_{i(\epsilon)}+1=n_{t(\epsilon)}$, by Lemma \ref{l35}, it follows that $\{d,f\}$ are the smallest and the second smallest positive weights. Moreover, $d$ and $f$ are strictly smaller than $a,b,c$, and $e$. By the second condition of Lemma \ref{l35} for $a$, we have $\Sigma_{p_0}=\{a,b,c\} \equiv \{-a,d,e\}=\Sigma_{p_1} \mod a$. This implies that $b=d$ or $c=d$, which is a contradiction.

Second, suppose that the case (2) holds. By Lemma \ref{l52}, $a,c,d,f<b$. By the second condition of Lemma \ref{l35} for $b$, we have $\{a,b,c\} \equiv \{-b,-d,f\} \mod b$. It follows that either
\begin{enumerate}[(a)]
\item $a+d=b$ and $c=f$.
\item $a=f$ and $c+d=b$.
\end{enumerate}

Assume that the case (a) holds. Since $c=f$, it follows that $f$ cannot be the smallest or the second smallest positive weight. Therefore $\{a,d\}$ are the smallest and the second smallest positive weights. Therefore, by Lemma \ref{l35}, $a<c$ and $d<c$. Next, since $\Sigma_{p_3}=\{-c,-e,-c(=-f)\}$ and the action is effective, $e \neq c$ and $m_{p_3}^c=2$. By (3) of Lemma \ref{l26}, at least two of $p_1$, $p_2$, and $p_3$ must have $m_p^c=2$. On the other hand, $2c>b$ where $b$ is the largest weight, $a<c$, $d<c$, and $e \neq c$. These imply that none of $a,b,d$, and $e$ can be a multiple of $c$ and this leads to a contradiction. Therefore, the case (a) cannot hold.

Assume that the case (b) holds. By Lemma \ref{l29}, we have
\begin{center}
$\displaystyle 0=\frac{1}{a(c+d)c}+\frac{1}{(-a)de}+\frac{1}{(-c-d)(-d)a}+\frac{1}{(-c)(-e)(-a)}=(\frac{1}{ac}+\frac{1}{ad})(\frac{1}{c+d}-\frac{1}{e})$.
\end{center}
This implies that $c+d=e$. Next, by Theorem \ref{t22},
\begin{center}
$\displaystyle \chi^0(M)=N^0=1=\frac{1}{(1-t^a)(1-t^c)(1-t^{c+d})}-\frac{t^a}{(1-t^a)(1-t^d)(1-t^{c+d})}+\frac{t^{c+2d}}{(1-t^a)(1-t^d)(1-t^{c+d})}-\frac{t^{a+2c+d}}{(1-t^a)(1-t^c)(1-t^{c+d})}$.
\end{center}
If we multiply the equation by the least common multiple of the denominators and simplify, we get $0=-t^c+t^{a+d}+t^{2c+d}-t^{a+c+2d}$. It follows that $c=a+d$. The weights at the fixed points are then
\begin{center}
$\Sigma_{p_0}=\{a,a+d,a+2d\},\Sigma_{p_1}=\{-a,d,a+2d\},\Sigma_{p_2}=\{-a-2d,-d,a\},\Sigma_{p_3}=\{-a-2d,-a-d,-a\}$.
\end{center}
This is the second case of Theorem \ref{t12}.

Third, suppose that the case (3) holds. By Lemma \ref{l52}, we have $a,b,e,f<c$. By the second condition of Lemma \ref{l35} for $c$, we have $\{a,b,c\} \equiv \{-c,-e,-f\} \mod c$. It follows that either
\begin{enumerate}[(a)]
\item $a+e=c$ and $b+f=c$.
\item $a+f=c$ and $b+e=c$.
\end{enumerate}

Suppose that the case (a) holds. Then we have $e=c-a$ and $f=c-b$. By Theorem \ref{t22}, we have
\begin{center}
$\displaystyle \chi^0(M)=1=\frac{1}{(1-t^a)(1-t^b)(1-t^{c})}-\frac{t^a}{(1-t^a)(1-t^d)(1-t^{e})}+\frac{t^{b+d}}{(1-t^b)(1-t^d)(1-t^{f})}-\frac{t^{c+e+f}}{(1-t^c)(1-t^e)(1-t^{f})}=[(\sum_{i=0}^{\infty}t^{ia})(\sum_{i=0}^{\infty}t^{ib})(\sum_{i=0}^{\infty}t^{ic})]-[t^a(\sum_{i=0}^{\infty}t^{ia})(\sum_{i=0}^{\infty}t^{id})(\sum_{i=0}^{\infty}t^{ie})]+[t^{b+d}(\sum_{i=0}^{\infty}t^{ib})(\sum_{i=0}^{\infty}t^{id})(\sum_{i=0}^{\infty}t^{if})]-[t^{c+e+f}(\sum_{i=0}^{\infty}t^{ic})(\sum_{i=0}^{\infty}t^{ie})(\sum_{i=0}^{\infty}t^{if})]$.
\end{center}
In the equation, constant terms match and $t^{ia}$ in the first bracket cancels out with $-t^{ia}$ in the second bracket for each $i>0$. Consider $t^b$ in the first bracket. Since the terms with smallest exponents in the third and the fourth brackets are $t^{b+d}$ and $t^{c+e+f}$, $t^b$ cannot cancel out with any term in those brackets. This implies that $t^b$ in the first bracket has to cancel out by either $-t^{a+d}$ or $-t^{a+e}$ in the second bracket. However, $a+e=c$ is the largest weight and $b<c$, and hence $t^b$ cancels out with $-t^{a+d}$. It follows that $b=a+d$. The weights at the fixed points are then
\begin{center}
$\Sigma_{p_0}=\{a,b,c\},\Sigma_{p_1}=\{-a,b-a,c-a\},\Sigma_{p_2}=\{-b,a-b,c-b\},\Sigma_{p_3}=\{-c,a-c,b-c\}$.
\end{center}
This is the first case of Theorem \ref{t12}.

Suppose that the case (b) holds. Since $a,d,f$ are the only weights whose edges $\epsilon$ satisfy $n_{i(\epsilon)}+1=n_{t(\epsilon)}$, it follows that one of $a$ and $f$ is at most the second smallest positive weight. By reversing the circle action if necessary (Remark \ref{r53}), we may assume that $a$ is at most the second smallest positive weight. By the second condition of Lemma \ref{l35}, it follows that $a<b$ and $a<e$. Since $a+f=c$ and $b+e=c$, we have $a<b<f<c$ and $a<e<f<c$. This also implies that $d$ is at most the second smallest positive weight. It follows that $d<b$ and $d<e$.

By Lemma \ref{l35} for $f=c-a$, we have $\Sigma_{p_2}=\{-b,-d,c-a(=f)\} \equiv \{-c+a(=-f),-c+b(=-e),-c\}=\Sigma_{p_3} \mod c-a(=f)$. Since $b<c$ and $c=a+(c-a)<b+(c-a)$, $-b \neq -c \mod c-a$. Therefore, it follows that $-b \equiv -c+b \mod c-a$ and $-d \equiv -c \mod c-a$, i.e., $b \equiv c-b \mod c-a$ and $d \equiv c \mod c-a$. Since $c=a+(c-a)<b+(c-a)$, it follows that $b=c-b$, i.e., $2b=c$. Moreover, since $d<c$ and $c-a>\frac{c}{2}$, we have that $d+(c-a)=c$, i.e., $a=d$. Then the weights at the fixed points are
\begin{center}
$\Sigma_{p_0}=\{a,b,2b(=c)\}, \Sigma_{p_1}=\{-a,a(=d),b(=b-c=e)\}, \Sigma_{p_2}=\{-a(=-d),-b,c-a(=f)\}, \Sigma_{p_3}=\{-2b(=-c),-b(=b-c=-e),a-c(=-f)\}$.
\end{center}
Since $m_{p_0}^b=2$, by (3) of Lemma \ref{l26}, either $m_{p_1}^b=2$ or $m_{p_2}^b=2$. However, since $a<b<f<2b$, this cannot hold. Therefore, the case (b) cannot hold.

Fourth, suppose that the case (4) holds. By Lemma \ref{l52}, $a,b,e,f<d$. By the second condition of Lemma \ref{l35} for $d$, we have $\Sigma_{p_1}=\{-a,d,e\} \equiv \{-b,-d,f\}=\Sigma_{p_2} \mod d$. It follows that one of the following holds:
\begin{enumerate}[(a)]
\item $a=b$ and $e=f$
\item $a+f=d$ and $e+b=d$.
\end{enumerate}
Since $a,d,f$ are the only weights whose edges $\epsilon$ satisfy $n_{i(\epsilon)}+1=n_{t(\epsilon)}$, it follows that $\{a,f\}$ are the smallest and the second smallest positive weights. In particular, $e$ and $b$ are strictly bigger than $a$ and $f$. It follows that $a \neq b$ by the first condition of Lemma \ref{l35} and the case (a) cannot hold. If the case (b) holds, we have $d=e+b>a+f=d$, which is a contradiction. The case (4) does not hold. \end{proof}

\begin{lem} \label{l56} The fourth case in Lemma \ref{l51} does not hold. \end{lem}

\begin{proof}
The only weights whose edge $\epsilon$ satisfy $n_{i(\epsilon)}+1=n_{t(\epsilon)}$ are $c$ and $f$. By Lemma \ref{l35}, it follows that $\{c,f\}$ are the smallest positive weight and the second smallest positive weight, and $a,b,d,e$ are strictly bigger and $c,f$. 


We show that $a \neq b$. For this, assume that $a=b$. Then $m_p^a=2$. Therefore, by (3) of Lemma \ref{l26}, $m_{p_1}^a=2$ or $m_{p_3}^a=2$. This implies that $d$ and $e$ are multiples of $a$. Then we have $m_{p_0}^a(-)=m_{p_1}^a(-)=0$ and $m_{p_1}^a(-)=m_{p_3}^a(-)=2$, which contradicts (3) of Lemma \ref{l26}.

Therefore, $a \neq b$. Without loss of generality, assume that $a>b$. By the second condition of Lemma \ref{l35}, we have
\begin{center}
$\Sigma_{p_0}=\{a,b,c\} \equiv \{-a,-b,f\}=\Sigma_{p_2} \mod a$.
\end{center}
This implies that either
\begin{enumerate}[(a)]
\item $2b=a$ and $c=f$, or
\item $b=f$ and $b+c=a$.
\end{enumerate}
However, $f$ is either the smallest or second smallest positive weight and hence $b \neq f$ by the first condition of Lemma \ref{l35}. Therefore, $2b=a$ and $c=f$.

If we reverse the circle action(Remark \ref{r53}) and apply the same argument to weights $d$ and $e$, we conclude that $d \neq e$. Without loss of generality, assume that $d>e$. As before, it follows that $d=2e$. The multisets of weights are therefore
\begin{center}
$\Sigma_{p_0}=\{2b,b,c\}, \Sigma_{p_1}=\{-c,2e,e\}, \Sigma_{p_2}=\{-2b,-b,c\}, \Sigma_{p_3}=\{-2e,-e,c\}$. 
\end{center}
Since $m_{p_0}^b=2$, by (3) of Lemma \ref{l26}, it follows that $e$ is a multiple of $b$. We then have $m_{p_0}^b(-)=m_{p_1}^b(-)=0$ and $m_{p_1}^a(-)=m_{p_3}^a(-)=2$, which contradicts (3) of Lemma \ref{l26}. \end{proof}

\section{6-dimension with 4 fixed points: the case that $\mathrm{Todd}(M)=0$}

In this section, we classify the fixed point data of a circle action on a 6-dimensional almost complex manifold $M$ with 4 fixed points, when the Todd genus of $M$ is 0. As in the case that $\mathrm{Todd}(M)=1$, we shall consider a multigraph associated to $M$ in the sense of Lemma \ref{l35}, and then classify the fixed point data in each case.

\begin{lem} \label{l61}
Let the circle act effectively on a 6-dimensional compact almost complex manifold with 4 fixed points, whose Todd genus is 0. Then $N^0=N^3=0$ and $N^1=N^2=2$. Let $n_{p_i}=1$ for $i=1,2$ and $n_{p_i}=2$ for $i=3,4$. By permuting $p_1$ and $p_2$ and by permuting $p_3$ and $p_4$ if necessary, there exist positive integers $a,b,c,d,e,f$ so that exactly one of the figures in Figure \ref{fig7} occurs as a multigraph associated to $M$ that satisfies the conditions in Lemma \ref{l35}. Alternatively, exactly one of the following holds for the multisets of weights at $p_i$:
\begin{enumerate}[(1)]
\item $\Sigma_{p_1}=\{-a,b,c\}, \Sigma_{p_2}=\{-c,a,d\}, \Sigma_{p_3}=\{-b,-e,f\}, \Sigma_{p_4}=\{-d,-f,e\}$.
\item  $\Sigma_{p_1}=\{-a,b,c\}, \Sigma_{p_2}=\{-d,a,e\}, \Sigma_{p_3}=\{-b,-e,f\}, \Sigma_{p_4}=\{-c,-f,d\}$.
\item $\Sigma_{p_1}=\{-a,b,c\}, \Sigma_{p_2}=\{-d,a,e\}, \Sigma_{p_3}=\{-b,-c,f\}, \Sigma_{p_4}=\{-e,-f,d\}$.
\item $\Sigma_{p_1}=\{-a,b,c\}, \Sigma_{p_2}=\{-d,e,f\}, \Sigma_{p_3}=\{-b,-c,a\}, \Sigma_{p_4}=\{-e,-f,d\}$.
\item $\Sigma_{p_1}=\{-a,b,c\}, \Sigma_{p_2}=\{-d,e,f\}, \Sigma_{p_3}=\{-b,-f,a\}, \Sigma_{p_4}=\{-c,-e,d\}$.
\item $\Sigma_{p_1}=\{-a,b,c\}, \Sigma_{p_2}=\{-d,e,f\}, \Sigma_{p_3}=\{-e,-f,a\}, \Sigma_{p_4}=\{-b,-c,d\}$.
\end{enumerate}
\end{lem}

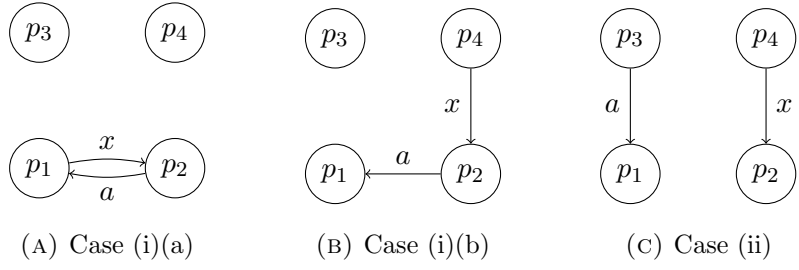
\begin{figure}
\begin{subfigure}[b][4cm][s]{.25\textwidth}
\centering
\vfill
\begin{tikzpicture}[state/.style ={circle, draw}]
\node[state] (a) {$p_1$};
\node[state] (b) [right=of a] {$p_2$};
\node[state] (c) [above=of a] {$p_3$};
\node[state] (d) [right=of c] {$p_4$};
\path (a) [->] [bend left =10]edge node[above] {$x$} (b);
\path (b) [->] [bend left =10] edge node [below] {$a$} (a);
\end{tikzpicture}
\vfill
\caption{Case (i)(a)}\label{fig6-1}
\vspace{\baselineskip}
\end{subfigure}\qquad
\begin{subfigure}[b][4cm][s]{.25\textwidth}
\centering
\vfill
\begin{tikzpicture}[state/.style ={circle, draw}]
\node[state] (a) {$p_1$};
\node[state] (b) [right=of a] {$p_2$};
\node[state] (c) [above=of a] {$p_3$};
\node[state] (d) [right=of c] {$p_4$};
\path (b) [->] edge node[above] {$a$} (a);
\path (d) [->] edge node[left] {$x$} (b);
\end{tikzpicture}
\vfill
\caption{Case (i)(b)}\label{fig6-2}
\vspace{\baselineskip}
\end{subfigure}\qquad
\begin{subfigure}[b][4cm][s]{.25\textwidth}
\centering
\vfill
\begin{tikzpicture}[state/.style ={circle, draw}]
\node[state] (a) {$p_1$};
\node[state] (b) [right=of a] {$p_2$};
\node[state] (c) [above=of a] {$p_3$};
\node[state] (d) [right=of c] {$p_4$};
\path (c) [->] edge node[left] {$a$} (a);
\path (d) [->] edge node[right] {$x$} (b);
\end{tikzpicture}
\vfill
\caption{Case (ii)}\label{fig6-3}
\vspace{\baselineskip}
\end{subfigure}\qquad
\caption{Cases in Lemma \ref{l61}.}\label{fig13}
\end{figure}

\begin{figure}
\begin{subfigure}[b][4cm][s]{.25\textwidth}
\centering
\vfill
\begin{tikzpicture}[state/.style ={circle, draw}]
\node[state] (a) {$p_1$};
\node[state] (b) [right=of a] {$p_2$};
\node[state] (c) [above=of a] {$p_3$};
\node[state] (d) [right=of c] {$p_4$};
\path (a) [->] [bend left =10]edge node[above] {$c$} (b);
\path (a) [->] edge node [left] {$b$} (c);
\path (b) [->] [bend left =10] edge node [below] {$a$} (a);
\path (b) [->] edge node [right] {$d$} (d);
\path (d) [->] [bend left =10]edge node [below] {$e$} (c);
\path (c) [->] [bend left =10]edge node [above] {$f$} (d);
\end{tikzpicture}
\vfill
\caption{(1)}\label{fig7-1}
\vspace{\baselineskip}
\end{subfigure}\qquad
\begin{subfigure}[b][4cm][s]{.25\textwidth}
\centering
\vfill
\begin{tikzpicture}[state/.style ={circle, draw}]
\node[state] (a) {$p_1$};
\node[state] (b) [right=of a] {$p_2$};
\node[state] (c) [above=of a] {$p_3$};
\node[state] (d) [right=of c] {$p_4$};
\path (b) [->] edge node[below] {$a$} (a);
\path (a) [->] edge node[pos=.2, left, sloped, rotate=270] {$c$} (d);
\path (a) [->]edge node [left] {$b$} (c);
\path (d) [->] edge node [right] {$d$} (b);
\path (b) [->]edge node[pos=.2, right, sloped, rotate=90] {$e$} (c);
\path (c) [->]edge node [above] {$f$} (d);
\end{tikzpicture}
\vfill
\caption{(2)}\label{fig7-2}
\vspace{\baselineskip}
\end{subfigure}\qquad
\begin{subfigure}[b][4cm][s]{.25\textwidth}
\centering
\vfill
\begin{tikzpicture}[state/.style ={circle, draw}]
\node[state] (a) {$p_1$};
\node[state] (b) [right=of a] {$p_2$};
\node[state] (c) [above=of a] {$p_3$};
\node[state] (d) [right=of c] {$p_4$};
\path (a) [->] [bend right =20]edge node[left] {$c$} (c);
\path (b) [->] edge node[below] {$a$} (a);
\path (a) [->] [bend left =20]edge node[left] {$b$} (c);
\path (d) [->] [bend left =20]edge node[left] {$d$} (b);
\path (b) [->] [bend left =20]edge node[left] {$e$} (d);
\path (c) [->] edge node[above] {$f$} (d);
\end{tikzpicture}
\vfill
\caption{(3)}\label{fig7-3}
\vspace{\baselineskip}
\end{subfigure}\qquad
\begin{subfigure}[b][4cm][s]{.25\textwidth}
\centering
\vfill
\begin{tikzpicture}[state/.style ={circle, draw}]
\node[state] (a) {$p_1$};
\node[state] (b) [right=of a] {$p_2$};
\node[state] (c) [above=of a] {$p_3$};
\node[state] (d) [right=of c] {$p_4$};
\path (a) [->] edge node[left] {$c$} (c);
\path (a) [->] [bend left =30]edge node[left] {$b$} (c);
\path (c) [->] [bend left =30]edge node[right] {$a$} (a);
\path (b) [->] edge node[left] {$f$} (d);
\path (b) [->] [bend left =30]edge node[left] {$e$} (d);
\path (d) [->] [bend left =30]edge node[right] {$d$} (b);
\end{tikzpicture}
\vfill
\caption{(4)}\label{fig7-4}
\vspace{\baselineskip}
\end{subfigure}\qquad
\begin{subfigure}[b][4cm][s]{.25\textwidth}
\centering
\vfill
\begin{tikzpicture}[state/.style ={circle, draw}]
\node[state] (a) {$p_1$};
\node[state] (b) [right=of a] {$p_2$};
\node[state] (c) [above=of a] {$p_3$};
\node[state] (d) [right=of c] {$p_4$};
\path (a) [->] edge node[left] {$b$} (c);
\path (b) [->] edge node[pos=.2, right, sloped, rotate=90] {$f$} (c);
\path (c) [->] [bend right =30]edge node[left] {$a$} (a);
\path (d) [->] edge node[right] {$d$} (b);
\path (b) [->] [bend right =30]edge node[right] {$e$} (d);
\path (a) [->] edge node[pos=.2, right, sloped, rotate=270] {$c$} (d);
\end{tikzpicture}
\vfill
\caption{(5)}\label{fig7-5}
\vspace{\baselineskip}
\end{subfigure}\qquad
\begin{subfigure}[b][4cm][s]{.25\textwidth}
\centering
\vfill
\begin{tikzpicture}[state/.style ={circle, draw}]
\node[state] (a) {$p_1$};
\node[state] (b) [right=of a] {$p_2$};
\node[state] (c) [above=of a] {$p_3$};
\node[state] (d) [right=of c] {$p_4$};
\path (a) [->] [bend left =10]edge node[pos=.1, left, sloped, rotate=270] {$b$} (d);
\path (a) [->] [bend right =10]edge node[pos=.1, right, sloped, rotate=270] {$c$} (d);
\path (c) [->]edge node [left] {$a$} (a);
\path (d) [->] edge node [right] {$d$} (b);
\path (b) [->][bend left =10]edge node[pos=.1, left, sloped, rotate=90] {$e$} (c);
\path (b) [->][bend right =10]edge node[pos=.1, right, sloped, rotate=90] {$f$} (c);
\end{tikzpicture}
\vfill
\caption{(6)}\label{fig7-6}
\vspace{\baselineskip}
\end{subfigure}\qquad
\caption{Completed multigraphs in Lemma \ref{l61}.}\label{fig7}
\end{figure}
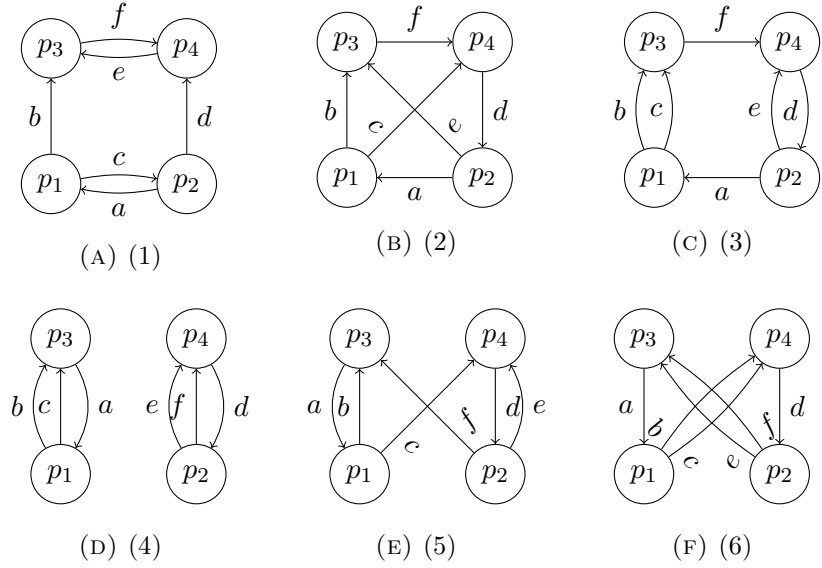

\begin{proof}
By Theorem \ref{t22}, $N^0=N^3$ and $N^1=N^2$. By Lemma \ref{l24}, there exists $i$ such that $N^i\neq0$ and $N^{i+1}\neq0$. Since $\mathrm{Todd}(M)=N^0=0$, it follows that $N^1=2$ and $N^2=2$. Let $p_i$ be fixed points with $n_{p_i}=1$ for $i=1,2$ and $n_{p_i}=2$ for $i=3,4$. 

Next, we consider a multigraph that satisfies conditions as in Lemma \ref{l35}. Let $-a$ be the negative weight at $p_1$ and $\epsilon_a$ its edge, i.e., $\epsilon_a$ is an edge with label (weight) $a$ and $t(\epsilon_a)=p_1$. Since we can associate a multigraph without any loop, by permuting $p_3$ and $p_4$ if necessary, two possibilities occur for $i(\epsilon_a)$:
\begin{enumerate}[(i)]
\item $i(\epsilon_a)=p_2$, i.e., $p_2$ has weight $+a$.
\item $i(\epsilon_a)=p_3$, i.e., $p_3$ has weight $+a$.
\end{enumerate}

First, suppose that $i(\epsilon_a)=p_2$. Next, let $-x$ be the negative weight at $p_2$. Let $\epsilon_x$ be an edge for $x$. By permuting $p_3$ and $p_4$ if necessary, two possibilities occur for $i(\epsilon_x)$.
\begin{enumerate}[(a)]
\item $i(\epsilon_x)=p_1$, i.e., $p_1$ has weight $+x$.
\item $i(\epsilon_x)=p_4$, i.e., $p_4$ has weight $+x$.
\end{enumerate}

Suppose that the case (a) holds. Since there is no loop, by permuting $p_3$ and $p_4$ necessary, the multigraph associated to $M$ must be as Figure \ref{fig7-1}. In this case, the weights are as in (1) after labelling the edges.


Suppose that the case (b) holds. In this case, if $\epsilon_u$ is an edge from $p_3$, it has to terminate at $p_4$. Completing a multigraph in the sense of Lemma \ref{l35}, there are two possibilities for the multigraphs; Figure \ref{fig7-2} and Figure \ref{fig7-3}. The weights in this case are as in (2) and (3), respectively.

Second, suppose that $i(\epsilon_a)=p_3$. Let $-x$ be the negative weight at $p_2$ and $\epsilon_x$ an edge for $x$. Then either $i(\epsilon_x)=p_1$ or $i(\epsilon_x)=p_4$. On the other hand,  the case that $i(\epsilon_x)=p_1$ is the same as the case (1)(b) by the horizontal symmetry if we permute $p_1$ and $p_2$ (and hence $a$ and $x$) and permute $p_3$ and $p_4$. Therefore, suppose that $i(\epsilon_x)=p_4$. Completing a multigraph that satisfies the conditions as in Lemma \ref{l35}, three possibilities occur; Figure \ref{fig7-4}, Figure \ref{fig7-5}, and Figure \ref{fig7-6}. The weights are as in (4), (5), and (6), respectively. \end{proof}

In the next lemma, we prove that as in the case that $\mathrm{Todd}(M)=1$, the maximum dimension of a connected component of $M^{\mathbb{Z}_l}$ is 2, where $l$ is the largest weight.

\begin{lem} \label{l62}
Let the circle act effectively on a 6-dimensional compact almost complex manifold $M$ with 4 fixed points, whose Todd genus is 0. Let $l$ be the largest weight among all the weights. Then $l>1$ and the dimension of a connected component of $M^{\mathbb{Z}_l}$ that contains a fixed point having the weight $l$ is two. In particular, any fixed point cannot have more than one weight that is a multiple of $l$.
\end{lem}

\begin{proof} We can associate a multigraph that satisfies the conditions in Lemma \ref{l35}. By Lemma \ref{l61}, one of the figures in Figure \ref{fig7} occurs as a multigraph associated to $M$. Since there exists an edge $\epsilon$ with $n_{i(\epsilon)}+1 \neq n_{t(\epsilon)}$, it follows that $l>1$. Suppose that a fixed point $p$ has the weight $l$. Consider the set $M^{\mathbb{Z}_l}$ of points fixed by the $\mathbb{Z}_l$-action. Let $Z$ be a connected component of $M^{\mathbb{Z}_l}$ that contains $p$. Since $l>1$ and the action is effective, $2 \leq \dim Z \leq 4$. Suppose that $\dim Z=4$. By Lemma \ref{l28}, $N_Z^i={2 \choose i}$. By changing $p_1$ and $p_2$, and $p_3$ and $p_4$ if necessary, this implies that
\begin{center}
$\{l,l\} \subset \Sigma_{p_1}, \{-l,l\} \subset \Sigma_{p_2}, \{-l,l\} \subset \Sigma_{p_3}, \{-l,-l\} \subset \Sigma_{p_4}$.
\end{center}
Let the remaining weight at $p_i$ be $-a,b,-c,d$, respectively, for some positive integers $a,b,c,d<l$. By Lemma \ref{l25}, $\Sigma_{p_i} \equiv \Sigma_{p_j} \mod l$ for any $i,j$. This implies that $a=c$, $b=d$, and $a+b=l$, i.e., the weights are
\begin{center}
$\Sigma_{p_1}=\{-a,l,l\}, \Sigma_{p_2}=\{-l,l-a,l\}, \Sigma_{p_3}=\{a-l,-l,l\}, \Sigma_{p_4}=\{-l,-l,a\}$.
\end{center}
By Theorem \ref{t22}, we have
\begin{center}
$\displaystyle 0=N^0=\frac{1}{(1-t^{-a})(1-t^l)^2}+\frac{1}{(1-t^{-l})(1-t^{l-a})(1-t^l)}+\frac{1}{(1-t^{a-l})(1-t^{-l})(1-t^l)}+\frac{1}{(1-t^{-l})^2(1-t^a)}=\frac{-t^a(1-t^{l-a})-t^l(1-t^a)+t^{2l-a}(1-t^a)+t^{2l}(1-t^{l-a})}{(1-t^a)(1-t^{l-a})(1-t^l)^2}=\frac{-t^a+t^{a+l}+t^{2l-a}-t^{3l-a}}{(1-t^a)(1-t^{l-a})(1-t^l)^2}$.
\end{center}
However, the last expression cannot be a constant, which leads to a contradiction. \end{proof}

The following lemma gives constraints on the weights at the fixed points.

\begin{lemma} \label{l63}
In Lemma \ref{l61}, Let $\Sigma_{p_i}=\{-w_{p_i}^1,w_{p_i}^2,w_{p_i}^3\}$ for $i=1,2$ and $\Sigma_{p_i}=\{-w_{p_i}^1,-w_{p_i}^2,w_{p_i}^3\}$ for $i=3,4$, where $w_{p_i}^j$ are positive integers. Then
\begin{center}
$\min\{w_{p_1}^1,w_{p_2}^1\}=\min\{w_{p_3}^1+w_{p_3}^2,w_{p_4}^1+w_{p_4}^2\}$.
\end{center}
Let $C=\{a,b,c,d,e,f\}$ be the collection of the positive weights and let $\{w_{p_i}^4,w_{p_i}^5,w_{p_i}^6\}=C\setminus\{w_{p_i}^1,w_{p_i}^2,w_{p_i}^3\}$ for each $i$. Then
\begin{center}
$\displaystyle \max_{i=1,2}\{w_{p_i}^1+\sum_{j=4}^6 w_{p_i}^j\}=\max_{i=3,4}\{w_{p_i}^1+w_{p_i}^2+\sum_{j=4}^6 w_{p_i}^j\}$.
\end{center} \end{lemma}

\begin{proof}
By Theorem \ref{t22},
\begin{center}
$\displaystyle \chi^0(M)=N^0=0=\sum_{i=1}^2 \frac{1}{(1-t^{-w_{p_i}^1})(1-t^{w_{p_i}^2})(1-t^{w_{p_i}^3})}+\sum_{i=3}^4 \frac{1}{(1-t^{-w_{p_i}^1})(1-t^{-w_{p_i}^2})(1-t^{w_{p_i}^3})}=\sum_{i=1}^2 \frac{-t^{w_{p_i}^1}}{(1-t^{w_{p_i}^1})(1-t^{w_{p_i}^2})(1-t^{w_{p_i}^3})}+\sum_{i=3}^4 \frac{t^{w_{p_i}^1+w_{p_i}^2}}{(1-t^{w_{p_i}^1})(1-t^{w_{p_i}^2})(1-t^{w_{p_i}^3})}$.
\end{center}
Next, in any case of Lemma \ref{l61}, the least common multiple of the denominators is $(1-t^a) \cdots (1-t^f)$. We multiply the equation above by $(1-t^a) \cdots (1-t^f)$ to have
\begin{center}
$0=\sum_{i=1}^2 \{-t^{w_{p_i}^1} \prod_{j=4}^6 (1-t^{w_{p_i}^j})\} + \sum_{i=3}^4 \{t^{w_{p_i}^1+w_{p_i}^2} \prod_{j=4}^6 (1-t^{w_{p_i}^j})\}$.
\end{center}
Comparing terms with the smallest exponents and with different signs, it follows that
\begin{center}
$\min\{w_{p_1}^1,w_{p_2}^1\}=\min\{w_{p_3}^1+w_{p_3}^2,w_{p_4}^1+w_{p_4}^2\}$.
\end{center}
Comparing terms with the biggest exponents and with different signs, it follows that
\begin{center}
$\displaystyle \max_{i=1,2}\{w_{p_i}^1+\sum_{j=4}^6 w_{p_i}^j\}=\max_{i=3,4}\{w_{p_i}^1+w_{p_i}^2+\sum_{j=4}^6 w_{p_i}^j\}$.
\end{center} \end{proof}

\begin{lem} \label{l64} The case (1) in Lemma \ref{l61} cannot hold. \end{lem}

\begin{proof} Assume on the contrary that the case (1) in Lemma \ref{l61} holds. Since the only edges $\epsilon_b$ for $b$ and $\epsilon_d$ for $d$ satisfy $n_{i(\epsilon)}+1=n_{t(\epsilon)}$, by the first condition of Lemma \ref{l35}, it follows that $\{b,d\}$ are the smallest and the second smallest weights. In particular, $a,c,e$, and $f$ are strictly bigger than $b$ and $d$. Next, by permuting $p_1$ and $p_2$ and by permuting $p_3$ and $p_4$ (by horizontal symmetry of the multigraph), and by reversing the circle action if necessary (Remark \ref{r53}), we may assume that $a$ is the biggest weight. By Lemma \ref{l62}, it follows that $c<a$. Next, by the second condition of Lemma \ref{l35} for $a$, we have that $\Sigma_{p_1}=\{-a,b,c\} \equiv \{-c,a,d\}=\Sigma_{p_2} \mod a$. Since $c>d$, we have that $2c=a$ and $b=d$. Then since $\Sigma_{p_1}=\{-2c,b,c\}$ and $\Sigma_{p_2}=\{-c,2c,b\}$, we have $m_{p_1}^c=m_{p_2}^c=2$. Therefore, by (3) of Lemma \ref{l26}, $m_{p_3}^c=2$ or $m_{p_4}^c=2$. Since $e,f>b$, this implies that $e$ and $f$ are multiples of $c$, i.e., $e=k_1c$ and $f=k_2c$ for some positive integers $k_1$ and $k_2$. Then we have $m_{p_i}^b(-)=1$ for all $i$. However, this contradicts the second statement of (3) of Lemma \ref{l26}. Therefore, the case (1) in Lemma \ref{l61} cannot hold. \end{proof}

\begin{lem} \label{l65} Assume that the case (2) in Lemma \ref{l61} holds. Then exactly one of the following holds for the multisets of weights at the fixed points:
\begin{enumerate}[(1)]
\item $\Sigma_{p_1}=\{-b-e,e,b\},\Sigma_{p_2}=\{-b-2e,b+e,e\},\Sigma_{p_3}=\{-e,-b-e,b+2e\},\Sigma_{p_4}=\{-b,-e,b+e\}$. This is the fourth case of Theorem \ref{t12}.
\item $\Sigma_{p_1}=\{-3b-c,b,c\},\Sigma_{p_2}=\{-2b-c,3b+c,3b+2c\},\Sigma_{p_3}=\{-b,-b-c,2b+c\},\Sigma_{p_4}=\{-c,-3b-2c,b+c\}$. This is the fifth case of Theorem \ref{t12}.
\item $\Sigma_{p_1}=\{-2b-c,b,c\},\Sigma_{p_2}=\{-b-c,2b+c,c\},\Sigma_{p_3}=\{-b,-c,2b+c\},\Sigma_{p_4}=\{-c,-2b-c,b+c\}$. This is the sixth case of Theorem \ref{t12}.
\end{enumerate} \end{lem}

\begin{proof} By reversing the circle action (Remark \ref{r53}), we may assume that exactly one of the following holds for the largest weight:
\begin{enumerate}
\item $a$ is the largest weight.
\item $b$ is the largest weight.
\item $c$ is the largest weight.
\item $d$ is the largest weight.
\end{enumerate}

First, suppose that $a$ is the largest weight. By Lemma \ref{l62}, we have that $b,c,d,e<a$. By Lemma \ref{l35}, $\Sigma_{p_1}=\{-a,b,c\}\equiv\{-d,a,e\}=\Sigma_{p_2} \mod a$. It follows that either
\begin{enumerate}[(a)]
\item $b+d=a$ and $c=e$, or
\item $b=e$ and $c+d=a$.
\end{enumerate}

Assume that the case (a) holds. We have
\begin{center}
$\Sigma_{p_1}=\{-b-d(=-a),b,c\},\Sigma_{p_2}=\{-d,b+d(=a),c\},\Sigma_{p_3}=\{-b,-c(=-e),f\},\Sigma_{p_4}=\{-c(=-e),-f,d\}$.
\end{center}
By Lemma \ref{l29}, we have
\begin{center}
$\displaystyle 0=\frac{1}{-(b+d)bc}+\frac{1}{-d(b+d)c}+\frac{1}{(-b)(-c)f}+\frac{1}{(-c)(-f)d}=(-\frac{1}{b+d}+\frac{1}{f})(\frac{1}{bc}+\frac{1}{cd})$.
\end{center}
It follows that $f=b+d$, i.e., $f=a$. By Lemma \ref{l63}, we have $\max\{b+2d+c+f(=a+d+e+f),d+b+c+f\}=\max\{2b+2c+2d(=b+c+a+d+e),2b+2c+d+f(=c+f+a+b+e)\}$. Subtracting $b+c+d$ from each term, this is equivalent to $\max\{d+f,f\}=\max\{b+c+d,b+c+f\}$. Since $d<f=a$, this implies that $d+f=b+c+f$, i.e., $d=b+c$. It follows that $a=f=b+d=2b+c$. Then the weights are
\begin{center}
$\Sigma_{p_1}=\{-2b-c,b,c\},\Sigma_{p_2}=\{-b-c,2b+c,c\},\Sigma_{p_3}=\{-b,-c,2b+c\},\Sigma_{p_4}=\{-c,-2b-c,b+c\}$.
\end{center}
This is the sixth case of Theorem \ref{t12}.

Assume that the case (b) holds. By Lemma \ref{l63}, we have $\min\{a,a-c(=d)\}=\min\{2b(=b+e),c+f\}$ and $\max\{2a-c+b+f(=a+d+e+f),a+b+f(=d+b+c+f)\}=\max\{2a+2b(=b+e+a+c+d),a+2b+c+f(=c+f+a+b+e)\}$. Since $a-c<a$, it follows that $a-c=\min\{2b(=b+e),c+f\}$. Since $a+b+f<a+2b+c+f$, it follows that $2a+b-c+f=\max\{2a+2b,a+2b+c+f\}$. On the other hand, if $a-c=c+f$, i.e., $a=2c+f$, then $2a+2b>a+2b+c+f$. Therefore, we have the following cases:
\begin{enumerate}[(i)]
\item $a-c=2b$ and $2a+b-c+f=2a+2b$.
\item $a-c=2b$ and $2a+b-c+f=a+2b+c+f$.
\item $a-c=c+f$ and $2a+b-c+f=2a+2b$.
\end{enumerate}

In case (i), we have that $a=2b+c$ and $f=b+c$. Then we have
\begin{center}
$\Sigma_{p_1}=\{-2b-c,b,c\},\Sigma_{p_2}=\{-2b,b,2b+c\},\Sigma_{p_3}=\{-b,-b,b+c\},\Sigma_{p_4}=\{-c,-b-c,2b\}$.
\end{center}
By the second condition of Lemma \ref{l35} for $f=b+c$, $\Sigma_{p_3}=\{-b,-b,b+c\}\equiv\{-c,-b-c,2b\}=\Sigma_{p_4} \mod b+c$. Since we must have $-b \equiv -c \mod b+c$, it follows that $b=c$. Since the action is effective, this implies that $b=1$. Then
\begin{center}
$\Sigma_{p_1}=\{-3,1,1\},\Sigma_{p_2}=\{-2,1,3\},\Sigma_{p_3}=\{-1,-1,2\},\Sigma_{p_4}=\{-1,-2,2\}$.
\end{center}
By Lemma \ref{l29}, we have
\begin{center}
$\displaystyle 0=\frac{1}{-3}+\frac{1}{-6}+\frac{1}{2}+\frac{1}{4} \neq 0$,
\end{center}
which is a contradiction.

Suppose that the case (ii) holds. Then $a=2b+c$ and $a=b+2c$, i.e, $b=c$. Since the action is effective, $c=1$. Then
\begin{center}
$\Sigma_{p_1}=\{-3,1,1\},\Sigma_{p_2}=\{-2,3,1\},\Sigma_{p_3}=\{-1,-1,1\},\Sigma_{p_4}=\{-1,-1,2\}$.
\end{center}
By Lemma \ref{l29}, this case is impossible.

Suppose that the case (iii) holds. From the latter we have $f=b+c$. Then we have
\begin{center}
$\Sigma_{p_1}=\{-b-3c,b,c\},\Sigma_{p_2}=\{-b-2c,b+3c,b\},\Sigma_{p_3}=\{-b,-b,b+c\},\Sigma_{p_4}=\{-c,-b-c,b+2c\}$.
\end{center}
By the second condition of Lemma \ref{l35} for $f=b+c$, $\Sigma_{p_3}=\{-b,-b,b+c\}\equiv\{-c,-b-c,2b\}=\Sigma_{p_4} \mod b+c$. Since we must have $-b \equiv -c \mod b+c$, it follows that $b=c$. Since the action is effective, this implies that $b=1$. Then
\begin{center}
$\Sigma_{p_1}=\{-4,1,1\},\Sigma_{p_2}=\{-3,1,4\},\Sigma_{p_3}=\{-1,-1,2\},\Sigma_{p_4}=\{-1,-2,3\}$.
\end{center}
As above, by Lemma \ref{l29} we get a contradiction.

Second, suppose that $b$ is the largest weight. Since $b,c$, and $e$ are the only weights whose corresponding edges satisfy $n_{i(\epsilon)}+1=n_{t(\epsilon)}$, this implies that $\{c,e\}$ are the smallest and the second smallest positive weights. In particular, $c$ and $e$ are strictly smaller than $a,b,d$, and $f$. By Lemma \ref{l62}, $a,f<b$. By Lemma \ref{l35}, $\Sigma_{p_1}=\{-a,b,c\}\equiv\{-b,-e,f\}=\Sigma_{p_3} \mod b$. Since $e<a$, it follows that $-a \equiv f \mod b$ and $c \equiv -e \mod b$, i.e., $a+f=b$ and $c+e=b$. Then we have $a+f=b=c+e<a+f$, which is a contradiction.

Third, suppose that $c$ is the largest weight. By an analogous argument as above, $\{b,e\}$ are the smallest and the second smallest positive weights, and $b$ and $e$ are strictly smaller than $a,c,d$, and $f$. By Lemma \ref{l62}, $a,b,d,f<c$. By Lemma \ref{l35}, $\Sigma_{p_1}=\{-a,b,c\}\equiv\{-c,-f,d\}=\Sigma_{p_4}\mod c$. Since $b<d$, $b \neq d \mod c$. Therefore, we have $a+d=c$ and $b+f=c$. By Lemma \ref{l63}, $\min\{a,c-a\}=\{b+e,c+(c-b)\}$. Since $b+e<c+(c-b)$, this implies that $\min\{a,c-a\}=b+e$.

Assume that $a=b+e$. With that $a=b+e$, $d=c-a=c-b-e$, and $f=c-b$, the second part of Lemma \ref{l63} implies that $c=2b+3e$. Then we have
\begin{center}
$\Sigma_{p_1}=\{-b-e,b,2b+3e\},\Sigma_{p_2}=\{-b-2e,b+e,e\},\Sigma_{p_3}=\{-b,-e,b+3e\},\Sigma_{p_4}=\{-2b-3e,-b-3e,b+2e\}$.
\end{center}
If we reverse the circle action (Remark \ref{r53}), then we have
\begin{center}
$\Sigma_{p_1}=\{-2b-3c,-b,b+e\},\Sigma_{p_2}=\{-b-e,-e,b+2e\},\Sigma_{p_3}=\{-b-3e,b,e\},\Sigma_{p_4}=\{-b-2e,b+3e,2b+3e\}$.
\end{center}
This is the fifth case of Theorem \ref{t12}.

Next, assume that $c-a=b+e$, i.e., $c=a+b+e$. Then we have
\begin{center}
$\Sigma_{p_1}=\{-a,b,a+b+e\},\Sigma_{p_2}=\{-b-e,a,e\},\Sigma_{p_3}=\{-b,-e,a+e\},\Sigma_{p_4}=\{-a-b-e,-a-e,b+e\}$.
\end{center}
By Lemma \ref{l35}, we have $\Sigma_{p_3}=\{-b,-e,a+e\}\equiv\{-a-b-e,-a-e,b+e\}=\Sigma_{p_4} \mod a+e$. It follows that $-e \equiv b+e \mod a+e$. Since $a>b$, this implies that $-e+(a+e)=b+e$, i.e., $a=b+e$. Then we have $m_{p_2}^{b+e}(-)=1$ for $i=1,2,4$, which contradicts the second statement of (3) of Lemma \ref{l26}.

Fourth, suppose that $d$ is the largest weight. By Lemma \ref{l62}, $a,e,c,f<d$. By reversing the circle action (Remark \ref{r53}), we may assume that $a \leq f$. Next, by Lemma \ref{l63}, $\min\{a,d\}=\min\{b+e,c+f\}$. Since $a \leq f$ and $a<d$, this implies that $a=b+e$. By Lemma \ref{l35}, $\Sigma_{p_2}=\{-d,a,e\}\equiv\{-c,-f,d\}=\Sigma_{p_4} \mod d$. This implies that either
\begin{enumerate}[(a)]
\item $a+c=d$ and $e+f=d$, i.e., $c=d-a$ and $f=d-e$.
\item $a+f=d$ and $e+c=d$, i.e., $f=d-a$ and $e=d-c$.
\end{enumerate}

Assume that the case (a) holds. Since $a+c=d$, $e+f=d$, and $a \leq f$, we have $e \leq c$. With $a=b+e$, $d=a+c=b+c+e$ and $f=d-e=b+c$. We have
\begin{center}
$\Sigma_{p_1}=\{-b-e,b,c\},\Sigma_{p_2}=\{-b-c-e,b+e,e\},\Sigma_{p_3}=\{-b,-e,b+c\},\Sigma_{p_4}=\{-c,-b-c,b+c+e\}$.
\end{center}
By Lemma \ref{l35} for $f=b+c$, $\Sigma_{p_3}=\{-b,-e,b+c\}\equiv\{-c,-b-c,b+c+e\}=\Sigma_{p_4} \mod b+c$. If $-b \equiv -c \mod b+c$ and $-e \equiv b+c+e \mod b+c$, since $e \leq c$, it follows that $b=c=e$. If $-e \equiv -c \mod b+c$, we have $c=e$. In either case we have $c=e$. Then we have
\begin{center}
$\Sigma_{p_1}=\{-b-e,b,e\},\Sigma_{p_2}=\{-b-2e,b+e,e\},\Sigma_{p_3}=\{-b,-e,b+e\},\Sigma_{p_4}=\{-e,-b-e,b+2e\}$.
\end{center}
This is the fourth case of Theorem \ref{t12}.

Next, assume that the case (b) holds. By Lemma \ref{l35}, we have $\Sigma_{p_3}=\{-b,-e,d-b-e(=f)\}\equiv\{e-d(=-c),b+e-d(=-f),d\}=\Sigma_{p_4} \mod d-b-e(=f)$. Note that since $a+f=d$ and $a\leq f$, $f \geq \frac{d}{2}$, where $d$ is the largest weight. Since $-b \equiv e-d \mod d-b-e$ and $d-b-e \equiv b+e-d \mod d-b-e$, we have $-e \equiv d \mod d-b-e$. Since $b+e=a\leq \frac{d}{2}$, $e<\frac{d}{2}$. This implies that $-e+2(d-b-e)=d$, i.e., $d=2b+3e$. Then we have
\begin{center}
$\Sigma_{p_1}=\{-b-e,b,2b+2e\},\Sigma_{p_2}=\{-2b-3e,b+e,e\},\Sigma_{p_3}=\{-b,-e,b+2e\},\Sigma_{p_4}=\{-2b-2e,-b-2e,2b+3e\}$.
\end{center}
By Lemma \ref{l35} for $c=2b+2e$, we have $\Sigma_{p_1}=\{-b-e,b,2b+2e\}\equiv\{-2b-2e,-b-2e,2b+3e\}=\Sigma_{p_4} \mod 2b+2e$, but this cannot hold. \end{proof}

\begin{lem} \label{l66} Assume that the case (3) in Lemma \ref{l61} holds. Then exactly one of the following holds for the multisets of weights at the fixed points:
\begin{enumerate}[(1)]
\item $\Sigma_{p_1}=\{-b-c,b,c\},\Sigma_{p_2}=\{-b-c-e,b+c,e\},\Sigma_{p_3}=\{-b,-e,b+e\},\Sigma_{p_4}=\{-e,-b-c,b+c+e\}$.
\item $\Sigma_{p_1}=\{-2b-c,b,c\},\Sigma_{p_2}=\{-b-c,2b+c,c\},\Sigma_{p_3}=\{-b,-c,2b+c\},\Sigma_{p_4}=\{-c,-2b-c,b+c\}$.
\end{enumerate}
Therefore, this is either the fourth case or the fifth case of Theorem \ref{t12}. \end{lem}

\begin{proof} By Lemma \ref{l29}, we have
\begin{center}
$\displaystyle 0=\frac{1}{(-a)bc}+\frac{1}{(-d)ae}+\frac{1}{(-b)(-c)f}+\frac{1}{(-e)(-f)d}=(-\frac{1}{a}+\frac{1}{f})(\frac{1}{bc}+\frac{1}{de})$.
\end{center}
It follows that $f=a$. By changing $b$ and $c$, we may assume that one of the following holds for the largest weight:
\begin{enumerate}
\item $a$ is the largest weight.
\item $b$ is the largest weight.
\item $d$ is the largest weight.
\item $e$ is the largest weight.
\end{enumerate}

First, suppose that $a$ is the largest weight. By Lemma \ref{l62}, we have that $b,c,d,e<a$. Next, by the second condition of Lemma \ref{l35} for $a$, $\Sigma_{p_1}=\{-a,b,c\}\equiv\{-d,a,e\}=\Sigma_{p_2} \mod a$. By permuting $b$ and $c$ if necessary, this implies that $b+d=a$ and $c=e$. Next, by Lemma \ref{l63}, $\min\{a,d\}=\min\{b+c,e+f\}$. Since $d<a$ and $b+c<e+f=c+a$, it follows that $d=b+c$. Then the weights are
\begin{center}
$\Sigma_{p_1}=\{-2b-c,b,c\},\Sigma_{p_2}=\{-b-c,2b+c,c\},\Sigma_{p_3}=\{-b,-c,2b+c\},\Sigma_{p_4}=\{-c,-2b-c,b+c\}$.
\end{center}
This is the fifth case of Theorem \ref{t22}.

Second, suppose that $b$ is the largest weight. By Lemma \ref{l62}, it follows that $c,a<b$. Next, the edges $\epsilon$ for $b,c$, and $e$ only satisfy $n_{i(\epsilon)}+1=n_{t(\epsilon)}$. On the other hand, $b$ is the largest weight. It follows that $\{c,e\}$ are the smallest and the second smallest positive weights. In particular, $c$ and $e$ are strictly smaller than $a,b,e$, and $f(=a)$. By the second condition of Lemma \ref{l35} for $b$, we have $\Sigma_{p_1}=\{-a,b,c\}\equiv\{-b,-c,f(=a)\}=\Sigma_{p_3} \mod b$. Since $c<a$, we cannot have $-a \equiv -c \mod b$. Therefore, we must have $-a \equiv a \mod b$ and $c \equiv -c \mod b$. These imply that $2a=2c=b$ and in particular, $a=c$, which is a contradiction. Therefore, $b$ (and hence $c$) cannot be the largest weight.

Third, suppose that $d$ is the largest weight. By Lemma \ref{l62}, we have $a,e<d$. By the second condition of Lemma \ref{l35} for $d$, $\Sigma_{p_2}=\{-d,a,e\} \equiv \{-e,-f(=-a),d\}=\Sigma_{p_4} \mod d$. It follows that either
\begin{enumerate}[(a)]
\item $a+e=d$.
\item $2a=2e=d$.
\end{enumerate}
In either case we have $a+e=d$. Next, by Lemma \ref{l63}, we have $\min\{a,d\}=\min\{b+c,e+f(=e+a)\}$. Since $a<d$ and $a<e+f(=e+a)$, we have $a=b+c$. Then the multisets of weights are
\begin{center}
$\Sigma_{p_1}=\{-b-c,b,c\},\Sigma_{p_2}=\{-b-c-e,b+c,e\},\Sigma_{p_3}=\{-b,-c,b+c\},\Sigma_{p_4}=\{-e,-b-c,b+c+e\}$.
\end{center}
If we let $b=A$, $c=B$, $e=C$, and $b+c=D$, then we have
\begin{center}
$\Sigma_{p_1}=\{-A-B,A,B\},\Sigma_{p_2}=\{-C-D,C,D\},\Sigma_{p_3}=\{-A,-B,A+B\},\Sigma_{p_4}=\{-C,-D,C+D\}$.
\end{center}
This is the fourth case of Theorem \ref{t12}.

Fourth, suppose that $e$ is the largest weight. Since $b,c$, and $e$ only have edges satisfying $n_{i(\epsilon)}+1=n_{t(\epsilon)}$ and $e$ is the largest weight, $\{b,c\}$ are the smallest and the second positive weights. In particular, $b$ and $c$ are strictly smaller than $a,d,e$, and $f$. By Lemma \ref{l62}, it follows that $a,d<e$. By the second condition of Lemma \ref{l35} for $e$, we have $\Sigma_{p_2}=\{-d,a,e\}\equiv\{-e,-f(=-a),d\}=\Sigma_{p_4} \mod e$. Therefore, either
\begin{enumerate}[(a)]
\item $d=f=a$.
\item $2d=e=2a$.
\end{enumerate}
Assume that $d=f=a$. Then we have $m_{p_i}^a=1$ for $i=1,3$ and $m_{p_i}^a=2$ for $i=2,4$, which contradicts (3) of Lemma \ref{l26}. Next, assume that $2d=e=2a$. The weights at $p_2$ are then $\{-d,d,2d\}$. Since the action is effective, this implies that $d=1$. However, $p_2$ cannot have weight $-1$ by the first condition of Lemma \ref{l35}. Therefore, the fourth case cannot hold. \end{proof}

\begin{lem} \label{l67} Assume that the case (4) in Lemma \ref{l61} holds. Then the multisets of weights at the fixed points are $\{-b-c,b,c\},\{-e-f,e,f\},\{-b,-c,b+c\}$, and $\{-e,-f,e+f\}$ for some positive integers $b,c,e$, and $f$, i.e., this is the fourth case of Theorem \ref{t12}. \end{lem}

\begin{proof}
By the horizontal symmetry of the multigraph associated, we may assume that $a \leq d$. By Lemma \ref{l63}, we have that $\min\{a,d\}=\min\{b+c,e+f\}$. First, assume that $b+c < e+f$. Then we have $a=b+c$. By Theorem \ref{t22}, we have
\begin{center}
$\displaystyle \chi^0(M)=N^0=0=\frac{1}{(1-t^{-b-c})(1-t^b)(1-t^c)}+\frac{1}{(1-t^{-d})(1-t^e)(1-t^f)}+\frac{1}{(1-t^{-b})(1-t^{-c})(1-t^{b+c})}+\frac{1}{(1-t^{-e})(1-t^{-f})(1-t^d)}=\frac{-t^{b+c}}{(1-t^{b+c})(1-t^b)(1-t^c)}+\frac{-t^d}{(1-t^{d})(1-t^e)(1-t^f)}+\frac{t^{b+c}}{(1-t^{b})(1-t^{c})(1-t^{b+c})}+\frac{t^{e+f}}{(1-t^{e})(1-t^{f})(1-t^d)}=\frac{-t^d+t^{e+f}}{(1-t^{d})(1-t^e)(1-t^f)}$.
\end{center}
It follows that $d=e+f$. This is the fourth case of Theorem \ref{t12}.

Second, assume that $b+c\geq e+f$. Then we have $d\geq a=e+f$. By the second condition of Lemma \ref{l35} for an edge with label $d$, we have $\Sigma_{p_2}=\{-d,e,f\} \equiv \{-e,-f,d\}=\Sigma_{p_4} \mod d$. Since $d \geq e+f$, either $2e=2f=d$, or $e+f=d$. In either case it follows that $d=e+f$. By an analogous argument as above, by Theorem \ref{t22}, it follows that $a=b+c$ if we consider $\chi^0(M)=0$. Again, this is the fourth case of Theorem \ref{t12}. \end{proof}

\begin{lem} \label{l68} Assume that the case (5) in Lemma \ref{l61} holds. Then the multisets of weights at the fixed points are
\begin{center}
$\Sigma_{p_1}=\{-a-b,a,b\},\Sigma_{p_2}=\{-c-d,c,d\},\Sigma_{p_3}=\{-a,-b,a+b\},\Sigma_{p_4}=\{-c,-d,c+d\}$
\end{center}
for some positive integers $a,b,c$, and $d$, i.e., this is the fourth case of Theorem \ref{t12}. \end{lem}

\begin{proof} By horizontal symmetry of the multigraph associated, we may assume that one of the following holds for the largest weight:
\begin{enumerate}
\item $a$ is the largest weight.
\item $b$ is the largest weight.
\item $c$ is the largest weight.
\end{enumerate}

First, suppose that $a$ is the largest weight. By Lemma \ref{l62}, it follows that $b,c,f<a$. By the second condition of Lemma \ref{l35} for $a$, $\Sigma_{p_1}=\{-a,b,c\}\equiv\{-b,-f,a\}=\Sigma_{p_3}\mod a$. Therefore, either
\begin{enumerate}[(a)]
\item $2b=a$ and $c+f=a$, or
\item $b+f=a$ and $b+c=a$.
\end{enumerate}

Assume that the case (a) holds. Suppose that $b>1$ and $b \neq c$. Then we have $m_{p_1}^b=m_{p_3}^b=2$. Therefore, by (3) of Lemma \ref{l26}, we must have $m_{p_2}^b=2$ or $m_{p_4}^b=2$. Since $b \neq c$, $f=a-c=2a-b \neq b$. This implies that $d$ and $e$ must by divisible by $b$, say $d=k_1 b$ and $e=k_2 b$ for some positive integers $k_i$, $i=1,2$. Then we have $m_{p_i}^b(-)=1$ for all $i$, which contradicts the second statement of (3) of Lemma \ref{l26}. Therefore, either $b=1$ or $b=c$. If $b=1$, then $a=2$ and hence $c=f=1$ from $c+f=a$. When $b=c$, by the effectiveness of the action, since $p_1$ has weights $\{-2b(=-a),b,b(=c)\}$, we have $b=1$. In either case we have $b=c=f=1$. Since the largest weight is 2, we have $f=1$. By Theorem \ref{t22}, we have
\begin{center}
$\displaystyle 0=\chi^0(M)=\frac{1}{(1-t^{-2})(1-t)^2}+\frac{1}{(1-t^{-d})(1-t)(1-t^e)}+\frac{1}{(1-t^{-1})^2 (1-t^{2})}+\frac{1}{(1-t^{-e})(1-t^{-1})(1-t^d)}=\frac{-t^2}{(1-t^{2})(1-t)^2}+\frac{-t^d}{(1-t^{d})(1-t)(1-t^e)}+\frac{t^2}{(1-t^{1})^2 (1-t^{2})}+\frac{t^{1+e}}{(1-t^{e})(1-t^{1})(1-t^d)}=\frac{-t^d+t^{1+e}}{(1-t^{e})(1-t^{1})(1-t^d)}$.
\end{center}
It follows that $d=e+1$. Since the largest weight is 2, this implies that $e=1$. The weights at the fixed points are
\begin{center}
$\Sigma_{p_1}=\{-2,1,1\},\Sigma_{p_2}=\{-2,1,1\},\Sigma_{p_3}=\{-1,-1,2\},\Sigma_{p_4}=\{-1,-1,2\}$.
\end{center}
This is the fourth case of Theorem \ref{t12}.

Assume that the case (b) holds. Then we have $c=f=a-b$. Since $a$ is the largest weight, $d \leq a$. Suppose that $d=a$. Since $d=a$ is the largest weight, by Lemma \ref{l62}, it follows that $e,a-b<d(=a)$. Then by the second condition of Lemma \ref{l35} for $d(=a)$, we have $\Sigma_{p_2}=\{-a(=-d),e,a-b(=f)\}\equiv\{b-a(=-c),-e,a(=d)\} \mod a(=d)$. Then either
\begin{enumerate}[(a)]
\item $e \equiv b-a \mod a$, i.e., $e-a=b-a$ and hence $b=e$.
\item $2e=a$ and $a-b \equiv b-a \mod a$, i.e., $2b=a$.
\end{enumerate} 
In either case we have that $b=e$. the weights are
\begin{center}
$\Sigma_{p_1}=\{-a,b,a-b\},\Sigma_{p_2}=\{-a,b,a-b\},\Sigma_{p_3}=\{b-a,-b,a\},\Sigma_{p_4}=\{b-a,-b,a\}$.
\end{center}
This is the fourth case of Theorem \ref{t12}. Next, suppose that $d<a$. By Lemma \ref{l63}, we have $\min\{a,d\}=\min\{b+f=b+(a-b)=a,e+f=e+(a-b)$. This implies that $d=e+(a-b)$. Then the weights are
\begin{center}
$\Sigma_{p_1}=\{-a,b,a-b\},\Sigma_{p_2}=\{-a+b-e,e,a-b\},\Sigma_{p_3}=\{b-a,-b,a\},\Sigma_{p_4}=\{-e,b-a,a-b+e\}$.
\end{center}
If we let $A=a-b$, $B=b$, $C=e$, and $D=a-b$, then we have
\begin{center}
$\Sigma_{p_1}=\{-A-B,A,B\},\Sigma_{p_2}=\{-C-D,C,D\},\Sigma_{p_3}=\{-A,-B,A+B\},\Sigma_{p_4}=\{-C,-D,C+D\}$.
\end{center}
This is the fourth case of Theorem \ref{t12}.

Second, suppose that $b$ is the largest weight. By Lemma \ref{l62} and Lemma \ref{l35} for $b$, from $\Sigma_{p_1}=\{-a,b,c\}\equiv\{-b,-f,a\}=\Sigma_{p_3} \mod b$, it follows that either
\begin{enumerate}[(a)]
\item $a=f$ and $c=a$, or
\item $2a=b$ and $c+f=b$.
\end{enumerate}
On the other hand, since only $b,c,e,f$ are the weights whose edges satisfy $n_{i(\epsilon)}+1=n_{t(\epsilon)}$ and $b$ is the largest weight, one of $c$ and $f$ is at most the second smallest positive weight. On the other hand, $a$ is strictly bigger than the second smallest positive weight since its edge $\epsilon_a$ satisfies $n_{i(\epsilon_a)}-1=n_{t(\epsilon_a)}$. This implies that $a$ is strictly bigger than at least one of $c$ and $f$. This implies that the case (a) is impossible. This also implies that when the case (b) holds, $c \neq a$ or $f \neq a$, and hence $c \neq a$ and $f \neq a$. Next, by reversing the circle action (Remark \ref{r53}), we may assume that $c \leq f$. Since $c+f=2a=b$ and $c \neq a$, this implies that $c < a=\frac{b}{2} <f=2a-c$. Next, by Lemma \ref{l63}, we have $\min\{a,d\}=\min\{b+f=2a+(2a-c),c+e\}$ and $\max\{a+d+e+f,a+b+c+d\}=\max\{b+c+d+e+f,a+b+c+e+f\}$. Since $a<b+f=2a+(2a-c)$, it follows that either $a=c+e$ or $d=c+e$. If $a=c+e$, we have $\max\{2c+d+4e,4c+d+3e\}=\max\{4c+d+5e,5c+6e\}$, which cannot hold. If $d=c+e$, we have $\max\{3a+2e,3a+2c+e\}=\max\{4a+c+2e,5a+e\}$ and hence $3a+2c+e=\max\{4a+c+2e,5a+e\}$. If $3a+2c+e=4a+c+2e$, we have $a<a+e=c<a$, which is a contradiction. If $3a+2c+e=5a+e$, we have $a=c<a$, which is a contradiction. Therefore, $b$ cannot be the largest weight.

Third, suppose that $c$ is the largest weight. Since the weights $b,c,e,f$ only have edges satisfying $n_{i(\epsilon)}+1=n_{t(\epsilon)}$ and $c$ is the largest weight, one of $b$ or $e$ is at most the second smallest positive weight. In particular, $a$ and $d$ are strictly bigger than the second smallest positive weight. By Lemma \ref{l62} and Lemma \ref{l35} for $b$, from $\Sigma_{p_1}=\{-a,b,c\}\equiv\{-c,-e,d\}=\Sigma_{p_3} \mod c$, it follows that either
\begin{enumerate}[(a)]
\item $a=e$ and $b=d$, or
\item $b+e=c$ and $a+d=c$.
\end{enumerate}
However, the case (a) cannot hold because of the discussion above. Therefore, the case (b) holds, and $e=c-b$ and $d=c-a$. Next, by Lemma \ref{l63}, $\min\{a,d=c-a\}=\min\{b+f,c+(c-b)=c+e\}$. Since $a<c$, this implies that either $a=b+f$ or $c-a=b+f$. Moreover, Lemma \ref{l63} says that $\max\{a+(c-a)+(c-b)+f,(c-a)+a+b+c\}=\max\{b+f+c+(c-a)+(c-b),c+(c-b)+a+b+f\}$, i.e., $\max\{2c-b+f,2c+b\}=\max\{3c-a+f,2c+a+f\}$. Since $2c-b+f<2c+a+f$, we have $2c+b=\max\{3c-a+f,2c+a+f\}$. From the minimum assume that $c-a=b+f$, i.e., $c=a+b+f$. Then from the maximum we have $2a+3b+2f=\max\{2a+3b+4f,3a+2b+3f\}$, which cannot hold. The equation $2c+b=\max\{3c-a+f,2c+a+f\}$ cannot also be satisfied when $a=b+f$. Therefore, $c$ cannot be the largest weight. \end{proof}

\begin{lem} \label{l69} Assume that the case (6) in Lemma \ref{l61} holds. Then the multisets of weights at the fixed points are
\begin{center}
$\Sigma_{p_1}=\{-b-c,b,c\},\Sigma_{p_2}=\{-b-c,b,c\},\Sigma_{p_3}=\{-b,-c,b+c\},\Sigma_{p_4}=\{-b,-c,b+c\}$
\end{center}
for some positive integers $b$ and $c$, i.e., this is the fourth case of Theorem \ref{t12}. \end{lem}

\begin{proof} By horizontal symmetry of the multigraph(see Figure \ref{fig7-6}), we may assume that one of the following holds for the largest weight:
\begin{enumerate}
\item $a$ is the largest weight.
\item $b$ is the largest weight.
\end{enumerate}

First, suppose that $a$ is the largest weight. By Lemma \ref{l62}, we have that $b,c,e,f<a$. By the second condition of Lemma \ref{l35} for $a$, $\Sigma_{p_1}=\{-a,b,c\}\equiv\{-e,-f,a\}=\Sigma_{p_3}\mod a$. By permuting $b$ and $c$ and by permuting $e$ and $f$ necessary, we may assume that $b+e=a$ and $c+f=a$, i.e., $e=a-b$ and $f=a-c$. Moreover, by reversing the circle action (Remark \ref{r53}), we may assume that $b+c \leq e+f$, i.e., $b+c \leq e+f=(a-b)+(a-c)$. It follows that $b+c \leq a$. Next, by Lemma \ref{l63}, we have $\min\{a,d\}=\min\{e+f=(a-b)+(a-c),b+c\}$. Since $d \leq a$ and $b+c \leq (a-b)+(a-c)$, from the minimum it follows that $d=b+c$. Next, as in Lemma \ref{l63}, we consider $\chi^0(M)=0$ in Theorem \ref{t22}, make the exponents in the denominators positive, multiply the equation by the least common multiple of the denominators $(1-t^a)(1-t^b)(1-t^c)(1-t^{b+c})(1-t^{a-b})(1-t^{a-c})$, and cancel out terms to have
\begin{center}
$0=-t^a+t^{a+b+c}-t^{3a-b-c}+t^{2b+c}+t^{b+2c}-t^{a+2b+c}-t^{a+b+2c}-t^{2b+2c}+t^{a+2b+2c}+t^{2a-b-c}-t^{2a+b+c}-t^{a+c}-t^{a+b}+t^{2a+c}+t^{2a+b}+t^{2a}$.
\end{center}
Since $b+c \leq a$, comparing the terms with smallest exponents and with different signs, we must have $a=b+c$. Then the multisets of weights at the fixed points are
\begin{center}
$\Sigma_{p_1}=\{-b-c,b,c\},\Sigma_{p_2}=\{-b-c,b,c\},\Sigma_{p_3}=\{-b,-c,b+c\},\Sigma_{p_4}=\{-b,-c,b+c\}$.
\end{center}
This is the fourth case of Theorem \ref{t12}.

Second, suppose that $b$ is the largest weight. By Lemma \ref{l62}, $a,c,d<b$. By the second condition of Lemma \ref{l35} for $b$, $\Sigma_{p_1}=\{-a,b,c\}\equiv\{-b,-c,d\}=\Sigma_{p_4} \mod b$. It follows that either
\begin{enumerate}[(a)]
\item $a=c$ and $c=d$, or
\item $a+d=b$ and $2c=b$.
\end{enumerate}

Assume that the case (a) holds. Since the edge $\epsilon_a$ for $a$ has $n_{i(\epsilon_a)}-1=n_{t(\epsilon_a)}$, by Lemma \ref{l35}, $a$ is strictly bigger than the second smallest positive weight. In particular, $a>1$. Moreover, since $c=a$ and $d=a$, we have $m_{p_1}^a=2$ and $m_{p_4}^a=2$. By (3) of Lemma \ref{l26}, it follows that $m_{p_2}^a=2$ or $m_{p_3}^a=2$. This implies that exactly one of $e$ and $f$ is a multiple of $a$. Then we have $m_{p_i}^a(-)=1$ for all $i$, which contradicts the second statement of (3) of Lemma \ref{l26}.

Assume that the case (b) holds. Then we have $\Sigma_{p_1}=\{-a,2c(=b),c\}$. If $c=1$, then the largest weight is $b=2$ and hence $a=1$, but the edge for $a$ has $n_{i(\epsilon_a)}-1=n_{t(\epsilon_a)}$, which does not satisfy the first condition of Lemma \ref{l35}. It follows that $c>1$. Since the action is effective, this implies that $a$ (and hence $d=2c-a$) cannot be a multiple of $c$. Then we have $m_{p_1}^c=2$ and $m_{p_4}^c=2$. By (3) of Lemma \ref{l26}, we must have $m_{p_2}^c=2$ or $m_{p_3}^c=2$. It follows that both $e$ and $f$ are multiples of $c$. Then we have $m_{p_i}^c(-)=0$ for $i=1,2$ and $m_{p_i}^c=2$ for $i=3,4$. In particular, there is no fixed point $p$ with $m_p^c(-)=1$ and this contradicts the second statement of (3) of Lemma \ref{l26}. Therefore, $b$ (and hence $c$, $e$, and $f$) cannot be the largest weight. \end{proof}

\section{Proof of Theorem \ref{t12}}

\begin{proof} [Proof of Theorem \ref{t12}] Quotienting out by the subgroup that acts trivially, we may assume that the action is effective. By Theorem \ref{t22}, $N^0=N^3$ and $N^1=N^2$, where $N^i$ is the number of fixed points that has exactly $i$ negative weights. By Lemma \ref{l24}, there exists $i$ such that $N^i \neq 0$ and $N^{i+1} \neq 0$. These imply that either
\begin{enumerate}
\item $N^i=1$ for $0 \leq i \leq 3$.
\item $N^1=N^2=2$ and $N^i=0$ otherwise.
\end{enumerate}
Note that $N^0=\mathrm{Todd}(M)$ is the Todd genus of $M$. In the first case, the Todd of $M$ is 1 and Theorem \ref{t12} follows from Lemma \ref{l51}, Lemma \ref{l53}, Lemma \ref{l54}, Lemma \ref{l55}, and Lemma \ref{l56}. In the second case, the Todd genus of $M$ is 0 and Theorem \ref{t12} follows from Lemma \ref{l61}, Lemma \ref{l64}, Lemma \ref{l65}, Lemma \ref{l66}, Lemma \ref{l67}, Lemma \ref{l68}, and Lemma \ref{l69}. \end{proof}

\end{document}